\RequirePackage{fix-cm}
\documentclass[smallcondensed]{svjour3}     
\smartqed  
\usepackage{graphicx, subfigure, chngcntr}
\usepackage{amsmath, amsfonts, amsopn, bm, multirow}

\newcommand\Hzdiv{H_0 (\div ; \Omega)}
\newcommand\p{\mathbf{p}}
\newcommand\q{\mathbf{q}}
\newcommand\tp{\tilde{\mathbf{p}}}

\newcommand\n{\mathbf{n}}
\newcommand\R{\mathbb{R}}
\newcommand\0{\mathbf{0}}

\newcommand\I{\I}
\newcommand\tY{\tilde{Y}}
\newcommand\tC{\tilde{C}}
\newcommand\tq{\tilde{\mathbf{q}}}
\newcommand\hp{\hat{\mathbf{p}}}

\renewcommand\H{\mathcal{H}}
\renewcommand\I{\mathcal{I}}
\newcommand\J{\mathcal{J}}
\renewcommand\L{\mathcal{L}}
\newcommand\N{\mathcal{N}}

\let\div\relax
\DeclareMathOperator{\div}{div}
\DeclareMathOperator{\tdiv}{\widetilde{\div}}
\DeclareMathOperator{\proj}{proj}
\DeclareMathOperator{\prox}{prox}
\DeclareMathOperator{\ran}{ran}
\DeclareMathOperator{\dom}{dom}
\DeclareMathOperator{\ri}{ri}
\DeclareMathOperator{\shr}{shr}
\DeclareMathOperator*{\argmin}{\arg\min}

\numberwithin{equation}{section}

\counterwithin{theorem}{section}

\spnewtheorem{proposition}[theorem]{Proposition}{\bfseries}{\itshape}
\spnewtheorem{lemma}[theorem]{Lemma}{\bfseries}{\itshape}
\spnewtheorem{corollary}[theorem]{Corollary}{\bfseries}{\itshape}
\spnewtheorem{remark}[theorem]{Remark}{\itshape}{\rmfamily}
\spnewtheorem{assumption}[theorem]{Assumption}{\itshape}{\rmfamily}

\usepackage{algorithm, algorithmic}

\usepackage[normalem]{ulem}
\usepackage{color}

\journalname{arXiv}

\begin{document}

\title{A Finite Element Nonoverlapping Domain Decomposition Method with Lagrange Multipliers for the Dual Total Variation Minimizations\thanks{The first author's work was supported by NRF grant funded by MSIT (NRF-2017R1A2B4011627)
and the second author's work was supported by NRF grant funded by the Korean Government (NRF-2015-Global Ph.D. Fellowship Program).}
}
\titlerunning{Domain Decomposition for Dual Total Variation}

\author{Chang-Ock Lee         \and
        Jongho Park
}
\authorrunning{C.-O. Lee and J. Park}
\institute{Chang-Ock Lee \at
              Department of Mathematical Sciences, KAIST, Daejeon 34141, Korea \\
              \email{colee@kaist.edu} 
           \and
           Jongho Park \at
           Department of Mathematical Sciences, KAIST, Daejeon 34141, Korea \\
           Tel.: +82-42-350-2790\\
           \email{jongho.park@kaist.ac.kr} 
}

\date{Received: date / Accepted: date}

\maketitle

\begin{abstract}
In this paper, we consider a primal-dual domain decomposition method for total variation regularized problems appearing in mathematical image processing.
The model problem is transformed into an equivalent constrained minimization problem by tearing-and-interconnecting domain decomposition.
Then, the continuity constraints on the subdomain interfaces are treated by introducing Lagrange multipliers.
The resulting saddle point problem is solved by the first order primal-dual algorithm.
We apply the proposed method to image denoising, inpainting, and segmentation problems
with either $L^2$-fidelity or $L^1$-fidelity.
Numerical results show that the proposed method outperforms the existing state-of-the-art methods.
\end{abstract}

\section{Introduction}
\label{Sec:Introduction}

After a pioneering work of Rudin~et al.~\cite{ROF:1992}, total variation minimization has been widely used in image processing.
In their work, authors proposed an image denoising model with the total variation regularizer, which is called Rudin--Osher--Fatemi~(ROF) model, as follows:
\begin{equation}
\label{ROF}
\min_{u \in BV(\Omega)} \left\{ \frac{\alpha}{2} \int_{\Omega} (u-f)^2 \,dx + TV(u) \right\},
\end{equation}
where $\Omega$ is an image domain, $f$ is a corrupted image, $\alpha > 0$ is a positive denoising parameter,
and $TV(u)$ is the total variation of $u$ defined as
\begin{equation*}
TV(u) = \sup \left\{ \int_{\Omega} u \div \q \,dx : \q \in (C_0^1 (\Omega))^2 , |\q(x)| \leq 1 \hspace{0.2cm}\forall x \in \Omega \right\},
\end{equation*}
where $|\cdot|$ denotes the Euclidean norm in $\mathbb{R}^2$.
The solution space $BV(\Omega)$ is the collection of $L^1$ functions with finite total variation. 
Thanks to the anisotropic diffusion property of the total variation term, the model~\eqref{ROF} effectively removes Gaussian noise and 
preserves edges and discontinuities of the image~\cite{SC:2003}.

Total variation minimization can be applied to not only image denoising, but also other interesting problems in image processing.
In~\cite{CS:2002}, an image inpainting model with the total variation regularizer was proposed.
The model in~\cite{CS:2002} is a simple modification of~\eqref{ROF}, which excludes the inpainting domain from the domain of integration of the fidelity term in~\eqref{ROF}.
In addition, \eqref{ROF} can be generalized to the image deconvolution problem by
replacing $u$ in the fidelity term by $Au$, where $A$ is either a specific convolution operator~\cite{CW:1998,YK:1996}.
In order to treat impulse noise, the $TV$-$L^1$ model, which uses $L^1$ fidelity instead of $L^2$ fidelity, was introduced in~\cite{CE:2005,Nikolova:2004}.
It is well-known that the $TV$-$L^1$ model preserves contrast of the image, while the conventional ROF model does not.
We also note that variational image segmentation problem can be represented as the total variation minimization by appropriate change of variables~\cite{CEN:2006}.
One can refer~\cite{CP:2016} for further results of total variation minimization.

This paper is concerned with domain decomposition methods~(DDMs) for such total variation regularized minimization problems.
DDMs are suitable for parallel computation since they solve a large scale problem by dividing it into smaller problems and treating them in parallel.
While DDMs for elliptic partial differential equations have been successfully developed over past decades, there have been relatively modest achievements in total variation minimization problems due to their own difficulties.
At first, the total variation term is nonsmooth and nonseparable.
Thus, the energy functional cannot be expressed as the sum of the local energy functionals in the subdomains in general.
Even more, the solution space $BV(\Omega)$ allows discontinuities of a solution on the subdomain interfaces,
so that it is difficult to impose appropriate boundary conditions to the local problems in the subdomains.

There have been several researches to overcome such difficulties and develop efficient DDMs for the total variation minimization~\cite{CTWY:2015,DCT:2016,Fornasier:2007,FLS:2010,FS:2009,HL:2013,HL:2015,LLWY:2016,LN:2017,LNP:2018,LPP:2019}.
The Gauss-Seidel and Jacobi type subspace correction methods for the ROF model were proposed in~\cite{Fornasier:2007,FLS:2010,FS:2009,LLWY:2016}, and they were extended to the case of mixed $L^1 / L^2$ fidelity in~\cite{HL:2013}.
However, in~\cite{LN:2017}, a counterexample was provided for convergence of subspace correction methods.
In~\cite{DCT:2016}, a convergent overlapping DDM for the convex Chan--Vese model~\cite{CEN:2006} was proposed.
Recently, a domain decomposition framework for the case of~$L^1$ fidelity was introduced in~\cite{LNP:2018}.
One of the effective ideas in DDMs for the total variation minimization is to consider the Fenchel--Rockafellar dual formulation of the model:
\begin{equation*}
\min_{\p \in H_0 (\div; \Omega)} \frac{1}{2\alpha} \int_{\Omega} (\div \p + \alpha f )^2 \,dx \hspace{0.5cm}
\textrm{subject to } |\p(x)| \leq 1 \hspace{0.2cm}\forall x \in \Omega
\end{equation*}
instead of the original one.
Here, $H_0 (\div ; \Omega)$ denotes the space of vector fields $\q$:~$\Omega \rightarrow \mathbb{R}^2$ such that
$\div \q \in L^2 (\Omega)$ and $\q \cdot \n = 0$ on $\partial \Omega$.
Consideration of the dual formulation resolves some difficulties mentioned above; the dual energy functional is separable, and the solution space $H_0 (\div ; \Omega)$ requires some regularity on the subdomain interfaces.
Even if there arises another difficulty of treating the inequality constraint $|\p(x)| \leq 1 \hspace{0.2cm}\forall x \in \Omega$, several successful researches have been done~\cite{CTWY:2015,HL:2015,LN:2017,LPP:2019}.

In this paper, we generalize the primal-dual DDM proposed in~\cite{LPP:2019} to a wider class of the total variation minimizations.
We consider a general model problem
\begin{equation}
\label{model_old}
\min_{u \in BV(\Omega)} \left\{ \alpha F(u) + TV(u) \right\}, 
\end{equation}
where $F$:~$BV(\Omega) \rightarrow \bar{\mathbb{R}}$ is a proper, convex, and lower semicontinuous functional satisfying additional properties such that $F$ is separable and simple in the sense that~\eqref{model_old} can be efficiently solved by the first order primal-dual algorithm~\cite{CP:2011}.
Such a class of total variation minimizations contains the ROF model, the $TV$-$L^1$ model, their inpainting variants, and the convex Chan--Vese model for image segmentation.
As we noted above, we treat the Fenchel--Rockafellar dual problem of~\eqref{model_old}:
\begin{equation}
\label{model_dual_old}
\min_{\p \in H_0 (\div ; \Omega)} \frac{1}{\alpha} F^*(\div \p) \hspace{0.5cm}
\textrm{subject to } |\p(x)| \leq 1 \hspace{0.2cm}\forall x \in \Omega ,
\end{equation}
where $F^*$ is the Legendre--Fenchel conjugate of $F$.
The image domain is decomposed into a number of disjoint subdomains $\left\{ \Omega_s \right\}_{s=1}^{\mathcal{N}}$, and the continuity along the subdomain interfaces is enforced.
By treating the continuity constraints on the subdomain interfaces by the method of Lagrange multipliers,
we obtain an equivalent saddle point problem.
Application of the Chambolle--Pock primal-dual algorithm~\cite{CP:2011} to the resulting saddle point problem yields our proposed method, that shows good performance for various total variation minimization problems such as image denoising, inpainting, and segmentation.

The rest of the paper is organized as follows.
We present the basic settings for design of DDM in Sect.~\ref{Sec:Basic}.
In Sect.~\ref{Sec:Model}, we state the abstract model problem which generalizes various problems in image processing such as image denoising, inpainting, and segmentation.
A convergent nonoverlapping DDM for the model problem is proposed in Sect.~\ref{Sec:Proposed}.
We apply the proposed method to several image processing problems and compare with the existing state-of-the-art methods in Sect.~\ref{Sec:Applications}.
We conclude the paper with remarks in Sect.~\ref{Sec:Conclusion}.

\section{Preliminaries}
\label{Sec:Basic}
In this section, we present the basic setting for design of DDM.
At first, we introduce notations that will be used throughout the paper.
Then, the discrete setting for the dual total variation minimization based on the finite element framework is provided.

\subsection{Notations}
Let $\H$ be the generic $n$-dimensional Hilbert space.
For $v=[v_1 , ..., v_n]^T, w=[w_1 , ..., w_n]^T \in \H$ and $1 \leq p < \infty$, the $p$-norm of $v$ is denoted by
\begin{equation*}
\| v \|_{p, \H} = \left( \sum_{i=1}^n |v_i|^p \right)^{\frac{1}{p}},
\end{equation*}
and the Euclidean inner product of $v$ and $w$ is denoted by
\begin{equation*}
\left< v, w\right>_{\H} = \sum_{i=1}^n v_i w_i .
\end{equation*}
We may drop the subscript $\H$ if there is no ambiguity.

In this paper, we use the symbol $^*$ as superscripts with two different meanings.
At first, for a convex functional $F$: $\H \rightarrow \bar{\mathbb{R}}$,
$F^*$:~$\H \rightarrow \bar{\mathbb{R}}$ denotes the Legendre--Fenchel conjugate of $F$, which is defined as
\begin{equation*}
F^*(v^*) = \sup_{v \in \H} \left\{ \left< v, v^* \right> - F(v) \right\}  
\end{equation*}
for every $v^* \in \H$.
On the other hand, when $A$: $\H \rightarrow \H$ is a linear operator on $\H$,
$A^*$ denotes the adjoint of $A$, that is,
\begin{equation*}
\left< v, Aw \right> = \left< A^* v, w \right>
\end{equation*}
for every $v, w \in \H$.
It is well-known that the matrix representation of $A^*$ is the transpose of the matrix representation of $A$.

For a convex functional $F$:~$\H \rightarrow \bar{\mathbb{R}}$, the effective domain of $F$, denoted by $\dom F$, is defined as
\begin{equation*}
\dom F = \left\{ u \in  \H : F(u) < \infty \right\}.
\end{equation*}
Also, for a convex subset $C$ of $\H$, we denote the relative interior of $C$ by $\ri C$, which is defined as the interior of $C$ when $C$ is regarded as a subset of its affine hull.
One may refer Sect.~6 of~\cite{Rockafellar:2015} for detailed topological properties of~$\ri C$.

For a subset $C$ of $\H$, we define the characteristic functional $\chi_C$: $\H \rightarrow \bar{\mathbb{R}}$ of $C$ by
\begin{equation*}
\chi_{C}(v) = \begin{cases} 0 & \textrm{if }v\in C, \\ \infty & \textrm{if } v \not\in C. \end{cases}
\end{equation*}
It is clear that the functional $\chi_C$ is convex if and only if $C$ is a convex subset of~$\H$.

\subsection{Fenchel--Rockafellar duality}
Throughout this paper, we use the notion of Fenchel--Rockafellar duality frequently.
For the sake of completeness, we provide its key features.
For more rigorous texts, readers may refer~\cite{CP:2016,Rockafellar:2015}, for instance.

Let $\H_1$ and $\H_2$ be finite-dimensional Hilbert spaces.
Consider the minimization problem of the form
\begin{equation}
\label{abstract_primal}
\min_{u \in \H_1} \left\{ F(u) + G(Ku) \right\},
\end{equation}
where $K$:~$\H_1 \rightarrow \H_2$ is a linear operator and $F$:~$\H_1 \rightarrow \bar{\mathbb{R}}$, $G$:~$\H_2 \rightarrow \bar{\mathbb{R}}$ are proper, convex, lower semicontinuous functionals.
We use the assumptions on~$F$ and $G$ presented in Corollary~31.2.1 of~\cite{Rockafellar:2015}.

\begin{assumption}
\label{Ass:equiv_general}
There exists $u \in \H_1$ such that $u \in \ri (\dom F)$ and $Ku \in \ri (\dom G)$.
\end{assumption}

\noindent
Under Assumption~\ref{Ass:equiv_general}, the following relations hold~\cite{Rockafellar:2015}:
\begin{subequations}
\begin{align}
\label{abstract_pd}
\min_{u \in \H_1} \left\{ F(u) + G(Ku) \right\} &= \min_{u \in \H_1} \max_{q \in \H_2} \left\{ \left< Ku, q\right>_{\H_2} + F(u) - G^*(q) \right\} \\
\label{abstract_dual}
&= - \left[ \min_{q \in \H_2} \left\{ F^*(-K^* q) + G^*(q) \right\} \right].
\end{align}
\end{subequations}
The saddle point problem in the right hand side of~\eqref{abstract_pd} is called the primal-dual formulation of~\eqref{abstract_primal}, and the minimization problem in \eqref{abstract_dual} is called the dual formulation of~\eqref{abstract_primal}.
Hence, it is enough to solve \eqref{abstract_pd} or \eqref{abstract_dual} instead of \eqref{abstract_primal} in many cases.

\subsection{Discrete setting}

Let $\Omega \subset \mathbb{R}^2$ be a rectangular image domain consisting of a number of rows and columns of pixels.
Since each pixel holds a value representing the intensity at a point, we may regard the image as a pixelwise constant function on $\Omega$.
We write the collection of all pixels in $\Omega$ as $\mathcal{T}$ and define the space $X$ for an image as
\begin{equation*}
X  = \left\{ u \ \in L^2 (\Omega) : u|_{T}\textrm{ is constant} \hspace{0.2cm}\forall T \in \mathcal{T} \right\}.
\end{equation*}
We note that $X \subset BV(\Omega)$.
We regard each pixel in~$\mathcal{T}$ as a square element so that $X$ becomes a piecewise constant square finite element space whose side length equals to~1.
It is clear that the functions
\begin{equation*}
\phi_T (x) = \begin{cases} 1 & \textrm{ if } x \in T , \\ 0 & \textrm{ if } x \not\in T , \end{cases} \hspace{0.5cm} T \in \mathcal{T}
\end{equation*}
form a basis for $X$.
For $u \in  X$ and $T \in \mathcal{T}$, let $(u)_T \in \mathbb{R}$ denote the degree of freedom~(dof) of $u$ associated with the basis function $\phi_T$.
Then, $u$ is represented as
\begin{equation*}
u = \sum_{T \in \mathcal{T}} (u)_T \phi_T.
\end{equation*}

To obtain a finite element discretization of the dual problem, the space $Y$ for the dual variables is defined by the lowest order Raviart--Thomas elements:
\begin{equation*}
Y = \left\{ \q \in \Hzdiv : \q|_{T} \in \mathcal{RT}_0(T) \hspace{0.2cm}\forall T \in \mathcal{T} \right\},
\end{equation*}
where $\mathcal{RT}_0(T)$ is the collection of the vector functions $\q$:~$T \rightarrow \R^2$ of the form
\begin{equation*}
\q (x_1 , x_2) = \begin{bmatrix} a_1 + b_1 x_1 \\ a_2 + b_2 x_2 \end{bmatrix}.
\end{equation*}
We notice that the divergence operator in the continuous setting is well-defined on $Y$ and $\div Y \subset X$.
Each dof of $Y$ is the value of the normal component on a pixel edge.
Let $\I$ be the set of indices of the basis functions for $Y$ and $\left\{\bm{\psi}_i \right\}_{i \in \I}$ be the basis.
For $\p \in Y$ and $i \in \I$, we denote the dof of $\p$ associated with the basis function $\bm{\psi}_i$ as $(\p)_i \in \mathbb{R}$; we have
\begin{equation*}
\p = \sum_{i \in \I} {(\p)_i \bm{\psi}_i}.
\end{equation*}
Throughout the paper, the $p$-norm of $\p \in Y$ is defined as the $p$-norm of the vector of dofs of~$\p$ for $1 \leq p < \infty$.
In this case, one may regard $\| \cdot \|_{2, Y}$ as a lumped $(L^2 (\Omega))^2$-norm with proper scaling; see Remark~2.2 of~\cite{LPP:2019}.

In order to treat the inequality constraints appearing in the dual formulation~\eqref{model_dual_old}, we define the convex subset $C$ of $Y$ by
\begin{equation}
\label{C}
C = \left\{ \p \in Y : |(\p)_i| \leq 1 \hspace{0.2cm}\forall i \in \I \right\}.
\end{equation}
The projection of $\p \in Y$ onto $C$ can be easily computed by
\begin{equation}
\label{proj_C}
(\proj_C \p)_i = \frac{(\p)_i}{\max \left\{ 1, |(\p)_i| \right\}} \hspace{0.5cm} \forall i \in \I.
\end{equation}

\subsection{Domain decomposition setting}
We decompose the image domain $\Omega$ into $\mathcal{N} = N \times N$ disjoint rectangular subdomains $\left\{ \Omega_s \right\}_{s=1}^{\mathcal{N}}$.
For two subdomains $\Omega_s$ and $\Omega_t$ ($s < t$) sharing a subdomain edge~(interface), let $\Gamma_{st} = \partial \Omega_s \cap \partial \Omega_t$ be the shared subdomain edge between them.
Also, we define the union of the subdomain interfaces $\Gamma$ by $\Gamma = \bigcup_{s<t} \Gamma_{st}$.
We assume that $\Gamma$ does not cut through any elements in~$\mathcal{T}$.

Now, we define the local function spaces in the subdomains.
For $s=1, ..., \mathcal{N}$, let $\mathcal{T}_s$ be the collection of all pixels in $\Omega_s$.
The local primal function space $X_s$ is defined by
\begin{equation*}
X_s  = \left\{ u \in L^2 (\Omega_s) : u|_{T}\textrm{ is constant} \hspace{0.2cm}\forall T \in \mathcal{T}_s \right\}.
\end{equation*}
Obviously, we have $X = \bigoplus_{s=1}^{\mathcal{N}} X_s$.
Furthermore, we define the natural restriction operator $R_s$:~$X \rightarrow X_s$ by
\begin{equation}
\label{R}
R_s u = u|_{\Omega_s}.
\end{equation}
Then, the adjoint $R_s^*$:~$X_s \rightarrow X$ of $R_s$ becomes the extension-by-zero operator
\begin{equation*}
(R_s^* u_s )_T = \begin{cases} (u_s)_T & \textrm{ if } T \subset \Omega_s , \\ 0 & \textrm{ if } T \not\subset \Omega_s , \end{cases} \hspace{0.2cm} T \in \mathcal{T}.
\end{equation*}
The local dual function spaces are defined in the tearing-and-interconnecting\linebreak sense~\cite{FR:1991}.
More precisely, we define the local dual function space $\tY_s$ as
\begin{equation*}
\tY_s = \left\{ \tq_s \in  H(\div ; \Omega_s) : \tq_s \cdot \n_s = 0 \textrm{ on } \partial \Omega_s \setminus \Gamma \textrm{, }
\tq_s |_{T} \in \mathcal{RT}_0 (T) \hspace{0.2cm}\forall T \in \mathcal{T}_s \right\}.
\end{equation*}
Note that the boundary condition $\tilde{\q}_s \cdot \n_s = 0$ is not imposed on $\Gamma \cap \partial \Omega_s$ for $\tY_s$.
Thus, $\tY_s$ has dofs on the pixel edges contained in $\partial \Omega_s \cap \Gamma$.
Let $\tilde{\I}_s$ be the set of indices of the basis functions for $\tY_s$.
Similarly to~\eqref{C}, the inequality-constrained subset $\tC_s$ of $\tY_s$ is defined by
\begin{equation*}
\tC_s = \left\{ \tp_s \in \tY_s : |(\tp_s)_i| \leq 1 \hspace{0.2cm}\forall i \in \tilde{\I}_s \right\}.
\end{equation*}
The orthogonal projection onto $\tC_s$ is computed as
\begin{equation*}
(\proj_{\tC_s} \tp_s )_i = \frac{(\tp_s)_i}{\max \left\{ 1, |(\tp_s)_i| \right\}} \hspace{0.5cm} \forall i \in \tilde{\I}_s.
\end{equation*}
Let $\tY$ and $\tC$ be the direct sums of the sets $\tY_s$'s and $\tC_s$'s, respectively.
By definition, functions in $\tY$ may have discontinuities on $\Gamma$.
Let $\I_{\Gamma}$ be the collection of dofs of $Y$ on $\Gamma$.
The jump operator $B$: $\tY \rightarrow \R^{|\I_{\Gamma}|}$ measures the magnitude of such discontinuities, 
that is, $B$ is defined as
\begin{equation*}
B\tp|_{\Gamma_{st}} = \tp_s \cdot \n_{st} - \tp_t \cdot \n_{st}, \hspace{0.5cm} s<t,
\end{equation*}
where $\tp_s = \tp|_{\Omega_s}$.
Clearly, there is a natural isomorphism between $Y$ and $\ker B \subset \tY$.
For later use, we provide an upper bound of the operator norm of $B$~\cite{LPP:2019}.

\begin{proposition}
\label{Prop:B_norm}
The operator norm of $B$\emph{:} $\tY \rightarrow \R^{|\I_{\Gamma}|}$ has a bound $\| B \|_{2}^2 \leq 2$.
\end{proposition}

Proposition~\ref{Prop:B_norm} will be used for the estimation of the range of the parameters in the proposed method.

\section{The model problem}
\label{Sec:Model}
The model problem we consider in this paper is the total variation regularized convex minimization problem defined on the image domain $\Omega$:
\begin{equation}
\label{model}
\min_{u \in X} \left\{ \J(u) := \alpha F(u) + TV(u) \right\},
\end{equation}
where $F$: $X \rightarrow \bar{\mathbb{R}}$ is a proper, convex, lower semicontinuous functional, $TV(u)$ is the discrete total variation of $u$ given by
\begin{equation}
\label{dTV}
TV(u) = \max_{\p \in C} \left< u, \div \p \right>_X,
\end{equation}
and $\alpha >0$ is a positive parameter.
In addition, we assume that $F$ satisfies the following three assumptions.

\begin{assumption}
\label{Ass:separable}
$F$ is separable in the sense that there exist proper, convex, lower semicontinuous local energy functionals $F_s$: $X_s \rightarrow \bar{\mathbb{R}}$ such that
$$
F(u) = \sum_{s=1}^{\mathcal{N}} F_s (R_s u),
$$
where $R_s$ is the restriction operator defined in~\eqref{R}.
\end{assumption}

\begin{assumption}
\label{Ass:simple}
For any $u \in X$ and $\sigma > 0$, the proximity operator $\prox_{\sigma F} (u)$ defined by
$$
\prox_{\sigma F} (u) = \argmin_{v \in X} \left\{ F(v) + \frac{1}{2\sigma} \| v- u \|_{2,X}^2 \right\}
$$
has a closed-form formula.
\end{assumption}

\begin{assumption}
\label{Ass:equiv}
$\ri ( \dom F ) \neq \emptyset$.
\end{assumption}

Assumption~\ref{Ass:separable} makes~\eqref{model} more suitable for designing DDMs.
It will be explained in Sect.~\ref{Sec:Proposed}.
Also, by Assumption~\ref{Ass:simple}, we are able to adopt the primal-dual algorithm~\cite{CP:2011} to solve~\eqref{model}.
Indeed, if Assumption~\ref{Ass:simple} holds, we can solve an equivalent primal-dual form of~\eqref{model},
\begin{equation}
\label{model_pd}
\min_{u \in X} \max_{\p \in Y} \left\{ -\left<u, \div \p \right>_X + \alpha F(u) - \chi_{C} (\p) \right\} ,
\end{equation}
by the primal-dual algorithm.
For the model problem~\eqref{model}, Assumption~\ref{Ass:equiv} is equivalent to Assumption~\ref{Ass:equiv_general} when $G(Ku) = TV(u)$ since $\ri (\dom (\| \cdot \|_{1, Y})) = Y$.
That is, Assumption~\ref{Ass:equiv} is essential for the primal-dual equivalence.
It is trivial that all the problems we will deal with in this paper satisfy Assumption~\ref{Ass:equiv}.

Next, we consider the dual formulation of \eqref{model}:
\begin{equation}
\label{model_dual}
\min_{\p \in Y} \left\{ \frac{1}{\alpha} F^*(\div \p) + \chi_{C} (\p) \right\}.
\end{equation}
From~\eqref{model_pd}, it is possible to deduce a relationship between solutions of the primal problem~\eqref{model} and the dual problem~\eqref{model_dual}:
\begin{equation} \begin{split}
\label{pd_relation}
0 &\in - \div \p + \alpha \partial F(u) , \\
0 &\in - \div^* u - \partial \chi_C (\p) .
\end{split} \end{equation}
As the standard discretization for the total variation minimization is the finite difference method, we provide a relation between the finite difference discretization and our finite element discretization.
The following proposition means that a solution of the finite difference discretization of~\eqref{model_old} can be recovered from a solution of~\eqref{model_dual}.

\begin{proposition}
\label{Prop:equiv}
Assume that the image domain~$\Omega$ consists of $M \times N$ pixels.
Let $\p^* \in Y$ be a solution of~\eqref{model_dual}.
If $u^* \in X$ and $\p^*$ satisfy the primal-dual relation~\eqref{pd_relation}, then~$u^*$ is a solution of the minimization problem
\begin{equation*}
\min_{u \in X} \left\{ \alpha F(u) + \left\| |Du| \right\|_1 \right\},
\end{equation*}
where $D u$ is the forward finite difference operator
\begin{equation} \begin{split}
\label{FFD}
(Du)_{ij}^1 &= \begin{cases} u_{i+1, j} - u_{ij} & \textrm{ if } i=1, ..., M-1, \\ 0 & \textrm{ if } i = M,\end{cases} \\
(Du)_{ij}^2 &= \begin{cases} u_{i, j+1} - u_{ij} & \textrm{ if } j=1, ..., N-1, \\ 0 & \textrm{ if } j = N\end{cases}
\end{split} \end{equation}
and $(|Du|)_{ij} = |(Du)_{ij}^1| + |(Du)_{ij}^2|$.
\end{proposition}
\begin{proof}
It is straightforward by the same argument as Theorem~2.3 of~\cite{LPP:2019}.
\qed\end{proof}

\begin{remark}
In Proposition~\ref{Prop:equiv}, $TV(u)$ equals to the finite difference anisotropic total variation $\| | Du | \|_1$ due to the structure of the constraint set~$C$ in~\eqref{C}.
On the other hand, one can make $TV(u)$ equal to the finite difference isotropic total variation by replacing~$C$ by an appropriate convex subset of~$Y$; see~\cite{LPP:2019} for details.
\end{remark}

\begin{remark}
\label{Rem:highRT}
We give a remark on the relation between the continuous problem~\eqref{model_old} and the discrete problem~\eqref{model}.
For simplicity, we assume that the resolution of the image is $N \times N$.
We fix the side length of $\Omega$ by~1 and introduce a side length parameter $h = 1/N$.
The space $X^h$ is defined in the same manner as in Sect.~\ref{Sec:Basic}.3.
Then the question we have is whether a solution $u^{h*} \in X^h$ of the minimization problem
\begin{equation*}
\min_{u^h \in X^h} \left\{ \alpha F(u^h ) + TV(u^h) \right\}
\end{equation*}
accumulates at a solution $u^* \in BV(\Omega)$ of~\eqref{model_old} in $BV(\Omega)$ as $h \rightarrow 0$.
However, since $TV(u^h)$ does not $\Gamma$-converge to the $BV$-seminorm in $BV(\Omega)$ as $h \rightarrow 0$~(see Example~4.1 of~\cite{Bartels:2012}), this is not true in general.
It means that the space $X^h$ may not be considered as a correct space for solving~\eqref{model_old}.
To avoid this uncomfortable situation, one may use higher order Raviart--Thomas elements to discretize the dual formulation~\cite{Bartels:2012,HHSNW:2018}.

Nevertheless, thanks to Proposition~\ref{Prop:equiv}, one can construct a sequence accumulating at $u^*$ in the $L^1 (\Omega)$-topology from $\{ u^{h*}\}_{h > 0}$.
For any $u^h \in X^h$, we have
\begin{equation}
\label{TVFD}
TV (u^h) = h^2 \left\| | \mathrm{D}^h \mathrm{dof}^h (u^h) | \right\|_1,
\end{equation}
where $\mathrm{dof}^h (u^h) \in \mathbb{R}^{N \times N}$ is the dofs of $u^h$ and $\mathrm{D}^h$:~$\mathbb{R}^{N \times N} \rightarrow \mathbb{R}^{N \times N} \times \mathbb{R}^{N \times N}$ is the forward finite difference operator defined similarly to~\eqref{FFD}.
As it can be shown without much difficulty the right hand side of~\eqref{TVFD} $\Gamma$-converges to the $BV$-seminorm in $L^1 (\Omega)$ as $h \rightarrow 0$~\cite{CLL:2011,WL:2011}, there exists an interpolation operator $I^h$:~$\mathbb{R}^{N \times N} \rightarrow L^1 (\Omega)$ such that $I^h \mathrm{dof}^h (u^{h*}) $ accumulates at $u^*$ in $L^1 (\Omega)$ as $h \rightarrow 0$.
We note that for $u^h \in X^h$, $I^h \mathrm{dof}^h (u^{h}) \not\in X^h$ in general.
\end{remark}

The model problem \eqref{model} occurs in various areas of mathematical image processing.
One of the typical examples is the image denoising problem.
In~\cite{ROF:1992}, authors proposed the well-known ROF model which consists of the $L^2$-fidelity term and the total variation regularizer.
In the ROF model, $F$ is given by
\begin{equation*}
F(u) = \frac{1}{2} \| u - f \|_{2, X}^2
\end{equation*}
and Assumptions~\ref{Ass:separable} and~\ref{Ass:simple} are satisfied with
\begin{eqnarray*}
F_s (u_s) &=& \frac{1}{2} \| u_s - f \|_{2, X_s}^2, \\
\prox_{\sigma F }(u) &=& \frac{u + \sigma f}{1 + \sigma} .
\end{eqnarray*}
In order to preserve the contrast of an image, the $TV$-$L^1$ model which uses the $L^1$-fidelity term was introduced in~\cite{CE:2005,Nikolova:2004}:
\begin{equation*}
F(u) =  \| u - f \|_{1, X}.
\end{equation*}
The local functionals $F_s$ and the proximity operator $\prox_{\sigma F}$ are readily obtained as follows:
\begin{eqnarray*}
F_s (u_s) &=&  \| u_s - f \|_{1, X_s}, \\
\prox_{\sigma F }(u) &=& f + \shr_{\Omega} (u - f, \sigma),
\end{eqnarray*}
where the elementwise shrinkage operator $\shr_{S}$:~$X \times \mathbb{R}_{>0} \rightarrow X$ on $S \subset \Omega$ is defined as follows~\cite{WYYZ:2008}:
\begin{equation}
\label{shrinkage}
\left( \shr_{S}(v , \sigma) \right)_T = \begin{cases} \max \left\{ |(v)_T| - \sigma , 0 \right\} \frac{(v)_T}{|(v)_T|} & \textrm{ if } T \subset S, \\
(v)_T & \textrm{ if } T \not\subset S,\end{cases}, \hspace{0.2cm} T \in \mathcal{T},
\end{equation}
with the convention $\frac{0}{|0|} = 0$.
For simplicity, $S$ is assumed to be a union of elements of~$\mathcal{T}$ in~\eqref{shrinkage}.

The models for the image denoising problem are easily extended to the image inpainting problem~\cite{CS:2002}.
Let $D \subset \Omega$ be the inpainting domain and $f \in X$ be the known part of an image.
We assume that $D$ does not cut through any elements in~$\mathcal{T}$.
We set $f = 0$ on $D$ for simplicity.
Also, let $A$:~$X \rightarrow X$ be the restriction operator onto $\Omega \setminus D$, that is, $Au = 0$ on $D$ for all $u \in X$.
Then, $F$ is given by
\begin{equation*}
F(u) = \frac{1}{2} \| Au - f \|_{2, X}^2 .
\end{equation*}
Note that Assumptions~\ref{Ass:separable} and~\ref{Ass:simple} are ensured with
\begin{eqnarray*}
F_s (u_s) &=& \frac{1}{2} \| A_s u_s - f \|_{2, X_s}^2, \\
\left( \prox_{\sigma F}(u) \right)_T &=& \begin{cases} \frac{(u)_T + \sigma (f)_T}{1 + \sigma} & \textrm{ if } T \subset \Omega \setminus D, \\ (u)_T & \textrm{ if } T \subset D, \end{cases}, \hspace{0.2cm} T \in T, 
\end{eqnarray*}
where $A_s u_s =  R_s AR_s^* u_s$ for $s= 1, ..., \mathcal{N}$.
One can easily check that $A = \bigoplus_{s=1}^{\mathcal{N}} A_s$.
Similarly to the denoising model, we may use the $L^1$ fidelity term instead of the $L^2$ fidelity term as follows:
\begin{equation*}
F(u) = \| Au - f \|_{1, X} .
\end{equation*}
In this case, $F_s$ and $\prox_{\sigma F}$ are given by
\begin{eqnarray*}
F_s (u_s) &=& \| A_s u_s - f \|_{1, X_s}, \\
\prox_{\sigma F}(u) &=& \shr_{\Omega \setminus D} (u-f, \sigma).  
\end{eqnarray*}

Another typical example is the image segmentation problem.
In~\cite{CEN:2006}, authors proposed a convex image segmentation model with the total variation regularizer as follows:
\begin{equation}
\label{CCV_old}
\min_{u \in X} \left\{ \alpha \Big( \left<u, (f - c_1)^2 \right>_X + \left< 1-u , (f - c_2)^2 \right>_X \Big) + \chi_{\left\{ 0\leq u \leq 1 \right\}}(u) + TV(u) \right\},
\end{equation}
where $f$ is a given image, $c_1$ and $c_2$ are predetermined intensity values.
Writing $g = (f - c_1)^2 - (f - c_2)^2$ in \eqref{CCV_old} yields the following simpler form:
\begin{equation}
\label{CCV}
\min_{u \in X} \left\{ \alpha \left< u, g \right>_X +  \chi_{\left\{ 0\leq u \leq 1 \right\}}(u) + TV(u) \right\} .
\end{equation}
Then, \eqref{CCV} is of the form~\eqref{model} with
\begin{equation*}
F(u) = \left< u, g \right>_X +  \chi_{\left\{ 0\leq u \leq 1 \right\}}(u),
\end{equation*}
and satisfies Assumptions~\ref{Ass:separable} and~\ref{Ass:simple} with
\begin{eqnarray*}
F_s (u_s) &=&  \left< u_s , g\right>_{X_s} + \chi_{\left\{ 0\leq u_s \leq 1 \right\}}(u_s) \\
\prox_{\sigma F }(u) &=& \proj_{\left\{ 0\leq \cdot \leq 1 \right\}}(u - \sigma g).
\end{eqnarray*}
Here, $\proj_{\left\{ 0\leq \cdot \leq 1 \right\}}$ can be computed pointwise like~\eqref{proj_C}.

We close this section by mentioning that the image deconvolution problem is not a case of~\eqref{model} in general.
For an operator $A$:~$X \rightarrow X$ defined by the matrix convolution with some matrix kernel,
$A \neq \bigoplus_{s=1}^{\mathcal{N}} R_s A R_s^*$ since the computation of $R_sA u$ needs values of $u$ in $\Omega_s$ as well as adjacent subdomains.
Consequently, Assumption~\ref{Ass:separable} cannot be satisfied for the deconvolution problem.

\section{Proposed method}
\label{Sec:Proposed}
In this section, we extend the primal-dual DDM for the ROF model introduced in~\cite{LPP:2019} to the more general model problem~\eqref{model}.
The continuity of a solution on the subdomain interfaces is imposed in the \textit{dual} sense, that is, it is imposed by the method of Lagrange multipliers.
As a result, we obtain an equivalent saddle point problem, which is solved by the first order primal-dual algorithm~\cite{CP:2011}.

We start the section by stating the following simple proposition, which means that
the Legendre--Fenchel conjugate is separable if the original functional is separable.

\begin{proposition}
\label{Prop:good}
Let $F$:~$X \rightarrow \bar{\mathbb{R}}$ be a proper, convex, lower semicontinuous functional satisfying Assumption~\ref{Ass:separable}.
Then, its Legendre--Fenchel conjugate $F^*$ also satisfies Assumption~\ref{Ass:separable}.
\end{proposition}
\begin{proof}
For $u^* = \bigoplus_{s=1}^{\mathcal{N}} u_s^* \in X$, we have
\begin{align*}
F^*(u^*) &= \sup_{u \in X} \left\{ \left< u, u^* \right>_{X} - F(u) \right\} \\
&= \sum_{s=1}^{\mathcal{N}} \sup_{u_s \in X_s} \left\{ \left< u_s, u_s^* \right>_{X_s} - F_s(u_s) \right\} \\
&= \sum_{s=1}^{\mathcal{N}} F_s^* (u_s^*).
\end{align*}
Letting $(F^*)_s = F_s^*$ for $s=1, ..., \mathcal{N}$ completes the proof.
\qed\end{proof}

Thanks to Proposition~\ref{Prop:good}, we can transform the dual model problem~\eqref{model_dual} to an equivalent constrained minimization problem
\begin{equation}
\label{DD_constrained}
\min_{\tp \in \tY} \sum_{s=1}^{\mathcal{N}} \left\{ \frac{1}{\alpha} F_s^*(\div \tp_s)  + \chi_{\tC_s} (\tp_s) \right\} \hspace{0.5cm}
\textrm{subject to } B\tp = 0.
\end{equation}
In order to treat the continuity constraint $B\tp = 0$, the method of Lagrange multipliers for~\eqref{DD_constrained} yields the saddle point formulation
\begin{equation}
\label{DD}
\min_{\tp \in \tY} \max_{\lambda \in \R^{|\I_{\Gamma}|}} \L (\tp, \lambda),
\end{equation}
where
\begin{equation*}
\L (\tp, \lambda) = \sum_{s=1}^{\mathcal{N}} \left\{ \frac{1}{\alpha}F_s^*(\div \tp_s) + \chi_{\tC_s}(\tp_s) \right\} + \left< B\tp , \lambda \right>_{\R^{|\I_{\Gamma}|}}.
\end{equation*}
The following proposition summarizes the equivalence between the dual model problem~\eqref{model_dual} and the resulting saddle point problem~\eqref{DD}.

\begin{proposition}
\label{Prop:DD_equiv}
If $\p^* \in C$ is a solution of \eqref{model_dual},
then $\tp^* = \bigoplus_{s=1}^{\mathcal{N}} \p^* |_{\Omega_s} \in \tC$ is a primal solution of \eqref{DD}.
Conversely, if $\tp^*$ is a primal solution of \eqref{DD}, then $\tp^* \in \ker B$.
Hence $\tp^* \in C$ and $\tp^*$ is a solution of \eqref{model_dual}.
\end{proposition}

Now, we are ready to propose the main algorithm of this paper.
In the case of ROF model, recovering a primal solution $u^*$ from the computed dual solution $\p^*$
can be easily done by the primal-dual relation $u^* = f + \div \p^* / \alpha$.
However, in the general case, the primal solution may not be obtained as in the ROF case since the primal-dual relation~\eqref{pd_relation} does not always give an explicit formula for $u^*$.
Instead, we consider an algorithm to obtain a primal solution $u^*$ and a dual solution $\p^*$ simultaneously.
We begin with the primal-dual algorithm~\cite{CP:2011} applied to~\eqref{DD}.
In each iteration of the primal-dual algorithm, we need to solve the local problems of the following form:
\begin{equation}
\label{DD_local_dual}
\min_{\tp \in \tY} \sum_{s=1}^{\mathcal{N}} \left\{ \frac{1}{\alpha}F_s^* (\div \tp_s ) + \chi_{\tC_s}(\tp_s) + \frac{1}{2\tau} \| \tp_s - \hp_s^{(n+1)} \|_{2, \tY_s}^2 \right\}
\end{equation}
for $\hp_s^{(n+1)} \in \tY_s$.
The term $\frac{1}{2\tau} \| \tp_s - \hp_s^{(n+1)} \|_{2, \tY_s}^2$ appears because we compute proximal descent/ascent in each step of the primal-dual algorithm.
To obtain a primal solution $u^{(n+1)}$ and a dual solution $\tp^{(n+1)}$ simultaneously, we replace \eqref{DD_local_dual} by the following primal-dual formulation of \eqref{DD_local_dual}:
\begin{equation}
\label{DD_local_pd}
\min_{u \in X} \max_{\tp \in \tY} \sum_{s=1}^{\mathcal{N}}
\left\{ - \left< u_s, \div \tp_s \right>_{X_s} + \alpha F_s (u_s) - \chi_{\tC_s}(\tp_s) - \frac{1}{2\tau} \| \tp_s - \hp_s^{(n+1)} \|_{2, \tY_s}^2 \right\}.
\end{equation}
There is another advantage to solve~\eqref{DD_local_pd} instead of \eqref{DD_local_dual}.
Differently from the ROF model, it is sometimes cumbersome to get an explicit formula for $F_s^*$,
which makes the design of a local solver difficult.
We note that an explicit formula for $F^*$ in the case of $F(u) = \| Au - f \|_{1,X}$ with nonsingular $AA^*$ is given in~\cite{DHN:2009}, but it is somewhat complicated.
However, considering \eqref{DD_local_pd} does not require an explicit formula for $F^*$.

Similarly to the ROF case, the solution pair $(u^{(n+1)}, \tp^{(n+1)})$ can be constructed by assembling the local solution pairs $(u_s^{(n+1)}, \tp_s^{(n+1)})$ in the subdomains,
i.e., $u^{(n+1)} = \bigoplus_{s=1}^{\mathcal{N}} u_s^{(n+1)}$ and $\tp^{(n+1)} = \bigoplus_{s=1}^{\mathcal{N}} \tp_s^{(n+1)}$.
The local solution pair $(u_s^{(n+1)}, \tp_s^{(n+1)})$ in the subdomain $\Omega_s$ is obtained by solving the local problem
\begin{equation}
\label{DD_local}
\min_{u_s \in X_s} \max_{\tp_s \in \tY_s}
\left\{ - \left< u_s, \div \tp_s \right>_{X_s} + \alpha F_s (u_s) - \chi_{\tC_s}(\tp_s) - \frac{1}{2\tau} \| \tp_s - \hp_s^{(n+1)} \|_{2, \tY_s}^2 \right\}
\end{equation}
and each local problem can be solved in parallel.
We will address how to solve~\eqref{DD_local} in Sect.~\ref{Sec:Applications} in detail.

In summary, the proposed algorithm is presented in Algorithm~\ref{Alg:DD}.
As in the ROF case, the range $\tau \sigma < 1/2$ comes from Proposition~\ref{Prop:B_norm}.

\begin{algorithm}[]
\caption{Primal-dual domain decomposition method for the model problem~\eqref{model}}
\begin{algorithmic}[]
\label{Alg:DD}
\STATE Choose $L > 2$, $\tau, \sigma > 0$ with $\tau \sigma = \frac{1}{L}$.
Let $\tp^{(0)} = \tp^{(-1)} = \0$ and $\lambda^{(0)} = 0$.
\FOR{$n=0,1,2,...$}
\STATE $\displaystyle \lambda^{(n+1)} = \lambda^{(n)} + \sigma B \left( 2\tp^{(n)} - \tp^{(n-1)} \right)$
\STATE $\displaystyle \hp^{(n+1)} = \tp^{(n)} - \tau B^* \lambda^{(n+1)}$
\FOR{$s=1,2,...,\mathcal{N}$ \textbf{in parallel}}
\STATE $\displaystyle (u_s^{(n+1)}, \tp_s^{(n+1)}) \in \arg \min_{u_s \in X_s} \max_{\tp_s \in \tY_s}
\bigg\{ - \left< u_s, \div \tp_s \right>_{X_s} + \alpha F_s (u_s)$ \\
\hspace{5.5cm} $\displaystyle - \chi_{\tC_s}(\tp_s) - \frac{1}{2\tau} \| \tp_s - \hp^{(n+1)}|_{\Omega_s} \|_{2, \tY_s}^2 \bigg\}$
\ENDFOR
\STATE $\displaystyle u^{(n+1)} = \bigoplus_{s=1}^{\mathcal{N}} u_s^{(n+1)}$
\STATE $\displaystyle \tp^{(n+1)} = \bigoplus_{s=1}^{\mathcal{N}} \tp_s^{(n+1)}$
\ENDFOR
\end{algorithmic}
\end{algorithm}

Next, we analyze convergence of the proposed method.
Since the sequence $\{ ( \tp^{(n)}, \lambda^{(n)} ) \}$ generated by Algorithm~\ref{Alg:DD} agrees with the one generated by the standard primal-dual algorithm, $O(1/n)$ ergodic convergence is guaranteed.

\begin{theorem}
\label{Thm:DD}
Let $\{ ( \tp^{(n)}, \lambda^{(n)} ) \}$ be the sequence generated by Algorithm~\ref{Alg:DD}.
Then, $(\tp^{(n)}, \lambda^{(n)})$ converges to a saddle point $(\tp^* , \lambda^*) $ of~\eqref{DD} and it satisfies
\begin{equation*}
\L (\tp_n , \lambda) - \L (\tp, \lambda_n) \leq \frac{1}{n} \left( \frac{1}{\tau} \| \tp - \tp^{(0)} \|_{2,\tY}^2 + \frac{1}{\sigma} \| \lambda - \lambda^{(0)} \|_{2, \mathbb{R}^{|\mathcal{I}_{\Gamma}|}}^2\right)
\end{equation*}
for any $\tp \in \tY$ and $\lambda \in \mathbb{R}^{|\mathcal{I}_{\Gamma}|}$, where $\tp_n = \frac{1}{n} \sum_{i=1}^n \tp^{(i)}$ and $\lambda_n = \frac{1}{n} \sum_{i=1}^n \lambda^{(i)}$.
\end{theorem}
\begin{proof}
Since the term
$$\sum_{s=1}^{\mathcal{N}} \left\{ \frac{1}{\alpha} F_s^* (\div \tp_s) + \chi_{\tC_s}(\tp_s) \right\}$$
in \eqref{DD} is convex,
by Theorem~5.1 in~\cite{CP:2016}, we get the desired result.
\qed\end{proof}

By Theorem~\ref{Thm:DD}, we ensure the convergence of the sequences $\{ \tp^{(n)} \}$ and $\{ \lambda^{(n)} \}$.
On the other hand, since the sequence~$\{ u^{(n)} \}$ is a kind of byproduct of the primal-dual algorithm, the convergence theory for the primal-dual algorithm developed in~\cite{CP:2011,CP:2016} does not ensure global convergence of~$\{ u^{(n)} \}$. 
Thus, we will prove that it tends to a solution of~\eqref{model}.
For the sake of convenience, we rewrite~\eqref{DD_local_pd} as the following more compact form:
\begin{equation}
\label{DD_local_new}
\min_{u \in X} \max_{\tp \in \tY} \left\{ -\left< u, \tdiv \tp \right>_{X} + \alpha F(u) - \chi_{\tC} (\tp) - \frac{1}{2\tau} \| \tp - \hp^{(n+1)} \|_{2, \tY}^2 \right\} ,
\end{equation}
where $\tdiv$: $\tY \rightarrow X$ is defined as $\tdiv \tp = \bigoplus_{s=1}^{\mathcal{N}} \div \tp_s$.
Note that if $\tp \in \ker B$, then $\tp \in Y$ so that $\tdiv \tp = \div \tp$.
Furthermore, since $\tY$ has dofs on $\Gamma$, we have $\ker \tdiv^* = \left\{ 0 \right\}$.
Then, we readily see that $u^{(n+1)}$ is a solution of the minimization problem
\begin{equation*}
\min_{u \in X} \left\{ \alpha F(u) + G_{n+1}(-\tdiv^* u) \right\},
\end{equation*}
where $G_{n}$ is defined as
\begin{equation*}
G_{n}(\tq) = \max_{\tp \in \tY} \left\{ \left< \tp, \tq \right>_{\tY} - \chi_{\tC} (\tp) - \frac{1}{2\tau}  \| \tp - \hp^{(n)} \|_{2, \tY}^2 \right\},
\end{equation*}
the Legendre--Fenchel conjugate of
\begin{equation*}
\chi_{\tC} (\tp) + \frac{1}{2\tau} \| \tp - \hp^{(n)} \|_{2, \tY}^2 .
\end{equation*}
At first, we verify the boundedness of $\left\{ G_n \right\}$.

\begin{lemma}
\label{Lem:G_bd}
There exist a finite functional $\overline{G}$\emph{:} $\tY \rightarrow \R$ and a coercive functional $\underbar{G}$\emph{:} $\tY \rightarrow \R$ such that
\begin{equation*}
\underline{G} \leq G_{n} \leq \overline{G} \hspace{0.5cm} \forall n \in \mathbb{Z}_{>0}.
\end{equation*}
\end{lemma}
\begin{proof}
We first see that for any $\tq \in \tY$, we have $\frac{\tq}{|\tq|} \in \tC$, where $|\tq|$ is the dof-wise absolute value of $\tq$ and the division is done dof-wise with convention $\frac{0}{|0|} = 0$.
An upper bound of $\left\{ G_n \right\}$ is obtained as follows:
\begin{align*}
G_n (\tq) &= \max_{\tp \in \tC} \left\{ \left< \tp , \tq \right>_{\tY} - \frac{1}{2\tau} \| \tp - \hp^{(n)} \|_{2, \tY}^2\right\} \\
&\leq  \max_{\tp \in \tC}  \left< \tp , \tq \right>_{\tY} \\
&= \left< \frac{\tq}{|\tq|}, \tq \right>_{\tY} = \|  \tq  \|_{1, \tY}.
\end{align*}
Take $\overline{G}(\tq) = \| \tq \|_{1, \tY}$, which is clearly finite.

Now, we find a lower bound of $\left\{ G_n \right\}$.
Note that $\hp^{(n)} = \tp^{(n-1)} - \tau B^* \lambda^{(n)}$ in Algorithm~\ref{Alg:DD}.
Since $\{ \tp^{(n)} \}$ and $\{ \lambda^{(n)} \}$ are convergent by Theorem~\ref{Thm:DD},
$\{ \hp^{(n)} \}$ is also convergent.
Hence, $\{ \hp^{(n)} \}$ is bounded, i.e.\ there exists a constant $L > 0$ such that
\begin{equation*}
\| \hp^{(n)} \|_{2, \tY}^2 \leq L \hspace{0.5cm} \forall n \in \mathbb{Z}_{>0}.
\end{equation*}
Then, we have
\begin{align*}
G_n (\tq) &= \max_{\tp \in \tC} \left\{ \left< \tp , \tq \right>_{\tY} - \frac{1}{2\tau} \| \tp - \hp^{(n)} \|_{2, \tY}^2\right\} \\
&\geq  \max_{\tp \in \tC} \left\{ \left< \tp , \tq \right>_{\tY} - \frac{1}{\tau} \left( \| \tp \|_{2, \tY}^2 + \| \hp^{(n)} \|_{2, \tY}^2 \right) \right\} \\
&\geq  \max_{\tp \in \tC} \left\{ \left< \tp , \tq \right>_{\tY} - \frac{1}{\tau} \left( \| \tp \|_{2, \tY}^2 + L \right) \right\} \\
&\geq \left< \frac{\tq}{|\tq|} , \tq \right>_{\tY} - \frac{1}{\tau} \left( \left\| \frac{\tq}{|\tq|} \right\|_{2, \tY}^2 + L \right) \\
&\geq \| \tq \|_{1, \tY} - \frac{1}{\tau} \left( \sum_{s=1}^{\mathcal{N}} |\tilde{\I}_s| + L \right).
\end{align*}
Note that $\tilde{\I}_s$ denotes the set of indices of the basis functions for $\tY_s$ and the value of $\sum_{s=1}^{\mathcal{N}} |\tilde{\I}_s|$ depends on the image size and the number of subdomains $\mathcal{N}$.
Take
\begin{equation*}
\underline{G}(\tq) =  \| \tq \|_{1, \tY} - \frac{1}{\tau} \left( \sum_{s=1}^{\mathcal{N}} |\tilde{\I}_s| + L \right),
\end{equation*}
which is coercive due to the term $ \| \tq \|_{1, \tY}$.
\qed\end{proof}

Next lemma provides a criterion for equi-coercivity of a collection of functionals.

\begin{lemma}
\label{Lem:coercive}
Let $\left\{ F_i \right\}$,~$i=1,2,\dots$, be a collection of functionals from $X$ into $\bar{\mathbb{R}}$.
If there exists a coercive functional $F_0$\emph{:} $X \rightarrow \bar{\mathbb{R}}$ such that $F_i \geq F_0$ for all $i$,
then $\left\{ F_i \right\}$ is equi-coercive,
that is, for every $t \geq 0$, there exists a compact subset $K_t$ of $X$ such that
\begin{equation*}
\left\{ u \in X : F_i (u) \leq t \right\} \subset K_t  \hspace{0.5cm} \textrm{for all } i=1,2,\dots.
\end{equation*}
\end{lemma}
\begin{proof}
Let $t \geq 0$.
Since $F_0$ is coercive, there exists a compact subset $K_t$ of $X$ such that
\begin{equation*}
\left\{ u \in X : F_0 (u) \leq t \right\} \subset K_t.
\end{equation*}
On the other hand, for every $i$, $F_i \geq F_0$ implies that
\begin{equation*}
\left\{ u \in X : F_i (u) \leq t \right\} \subset \left\{ u \in X : F_0 (u) \leq t \right\}.
\end{equation*}
Therefore, $\left\{ F_i \right\}$ is equi-coercive.
\qed\end{proof}

Now, with the help of the lemmas above, we state the main theorem, which ensures that the sequence $\{ u^{(n)} \}$ generated by Algorithm~\ref{Alg:DD} approaches to solutions of \eqref{model}.

\begin{theorem}
\label{Thm:DD_u}
Let $\{ u^{(n)} \}$ be the sequence generated by Algorithm~\ref{Alg:DD}.
Then, $\{ u^{(n)} \}$ is bounded and every limit point of $\{ u^{(n)} \}$ is a solution of \eqref{model}.
\end{theorem}
\begin{proof}
Recall that $u^{(n)}$ is a solution of the minimization problem
\begin{equation*}
\min_{u \in X} \left\{  \alpha F(u) + G_{n}(-\tdiv^* u) \right\}.
\end{equation*}
Since $F$ is proper, we may choose $u_0 \in X$ with $F(u_0) < \infty$.
By Lemma~\ref{Lem:G_bd}, we have $$t = \alpha F(u_0) + \overline{G}(-\tdiv^* u_0) < \infty.$$
Thanks to the minimization property of $u^{(n)}$, we get
\begin{align*}
\alpha F(u^{(n)}) + G_n (- \tdiv^* u^{(n)}) &\leq \alpha F(u_0) + G_n (- \tdiv^* u_0) \\
&\leq \alpha F(u_0) + \overline{G} (- \tdiv^* u_0) = t.
\end{align*}
By Lemmas~\ref{Lem:G_bd} and~\ref{Lem:coercive}, $\left\{ G_n \right\}$ is equi-coercive,
that is, there exists a compact subset $K_t$ of $\tY$ independent of $n$ such that
\begin{equation*}
\left\{ \tq \in \tY : G_n (\q ) \leq t \right\} \subset K_t .
\end{equation*}
Thus, $G_n (-\tdiv^* u^{(n)}) \leq t$ implies that $-\tdiv^* u^{(n)} \in K_t \cap \ran\tdiv^*$.
Since $\ker \tdiv^* = \left\{ 0 \right\}$, the map $-\tdiv^*$ is a continuous isomorphism between $X$ and $\ran\tdiv^*$.
Therefore, we can deduce that
\begin{equation*}
u^{(n)} \in (-\tdiv^*)^{-1}(K_t \cap \ran\tdiv^*),
\end{equation*}
which is a compact subset of $X$ independent of $n$.
This implies that $\{ u^{(n)} \}$ is bounded.

Now, we may refine $\{ u^{(n)} \}$ so that it converges to its limit point $u^*$.
Since $(u^{(n)} , \tp^{(n)})$ is a solution of a saddle point problem~\eqref{DD_local_new}, it satisfies
\begin{equation*}
\frac{\tdiv \tp^{(n)}}{\alpha} \in \partial F(u^{(n)}).
\end{equation*}
By Theorem~\ref{Thm:DD}, $\{ \tp^{(n)} \}$ converges to $\tp^*$.
As $u^{(n)} \rightarrow u^*$ and $\frac{\tdiv \tp^{(n)}}{\alpha} \rightarrow \frac{\tdiv \tp^*}{\alpha}$,
by closedness of the graph of $\partial F$~(See Theorem~24.4 in~\cite{Rockafellar:2015}), we get
\begin{equation*}
\frac{\tdiv \tp^{*}}{\alpha} \in \partial F(u^{*}).
\end{equation*}
By Proposition~\ref{Prop:DD_equiv}, $\tp^* \in \ker B$, and hence $\tp^* \in Y$.
We obtain the relation
\begin{equation*}
0 \in - \div \tp^* + \alpha \partial F (u^*) .
\end{equation*}
From the facts that $\tp^*$ is a solution of~\eqref{model_dual} (See Proposition~\ref{Prop:DD_equiv}) and $(u^* , \tp^*)$ satisfies the relation~\eqref{pd_relation}, we conclude that $u^*$ is a solution of~\eqref{model}.
\qed\end{proof}

As a direct consequence of of Theorem~\ref{Thm:DD_u}, we get the following result.

\begin{corollary}
\label{Cor:DD_u}
Let $\{ u^{(n)} \}$ be the sequence generated by Algorithm~\ref{Alg:DD}.
Then, $\{ \J (u^{(n)}) \}$ converges to $\J (u^*)$, where $u^*$ is a solution of~\eqref{model}.
\end{corollary}

\section{Applications}
\label{Sec:Applications}
In this section, we apply our proposed DDM to various total variation based image processing problems mentioned in Sect.~\ref{Sec:Model}.
The proposed method was implemented in ANSI~C with OpenMPI and compiled by Intel Parallel Studio~XE.
All the computations were done on a cluster composed of seven machines, where each machine has two Intel Xeon SP-6148 CPUs~(2.4GHz, 20C), 192GB memory, and the operating system CentOS~7.4 64bit.

To emphasize efficiency of the proposed method as a parallel algorithm, we compare the wall-clock time of the proposed method with the primal-dual algorithm~\cite{CP:2011} for the full dimension problem~\eqref{model_dual}.
The wall-clock time measures the total elapsed time including the communication time.

At each iteration of Algorithm~\ref{Alg:DD}, we solve the local saddle point problems of the form
\begin{equation}
\label{local}
\min_{u_s \in X_s} \max_{\tp_s \in \tY_s} \left\{ -\left< u_s, \div \tp_s \right>_{X_s} + \alpha F_s (u_s )- G_s^* (\tp_s) \right\} ,
\end{equation}
where $F_s$ is given in Assumption~\ref{Ass:separable} and
\begin{equation*}
G_s^* (\tp_s) = \chi_{\tC_s} (\tp_s) + \frac{1}{2\tau} \| \tp_s - \hp_s^{(n+1)} \|_{2, \tY_s}^2.
\end{equation*}

The primal-dual algorithm for \eqref{local} consists of computation of the proximity operators of $\alpha F_s (u_s)$ and $G_s^* (\tp_s)$,
that is,
\begin{eqnarray*}
u_s &=& \prox_{\sigma_0 \alpha F_s} (\bar{u}_s), \\
\tp_s &=& \prox_{\tau_0 G_s^*} (\bar{\p}_s),
\end{eqnarray*}
for some $\sigma_0 , \tau_0 > 0$.
For more details, we refer readers to see~\cite{CP:2011}.
The proximity operator of $G_s^*(\tp_s)$ is computed easily as follows:
\begin{align*}
\prox_{\tau_0 G_s^*} (\bar{\p}_s) &= \argmin_{\tp_s \in \tY_s} \left\{ \chi_{\tC_s} (\tp_s) + \frac{1}{2\tau} \| \tp_s -\hp_s^{(n+1)}  \|_{2, \tY_s}^2 + \frac{1}{2\tau_0} \| \tp_s - \bar{\p}_s \|_{2, \tY_s}^2 \right\} \\
&= \proj_{\tC_s} \left( \frac{\tau \bar{\p}_s + \tau_0 \hp_s^{(n+1)}}{\tau + \tau_0}\right).
\end{align*}
Computation of the proximity operator of $\alpha F_s (u_s)$ depends on the problem to solve.
Thus, we will give details in each subsection.
 
 Since $G_s^* $ is uniformly convex with parameter $1/\tau$, we are able to adopt the $O(1/n^2)$ convergent primal-dual algorithm~\cite[Algorithm~2]{CP:2011} for the local problems.
Such acceleration of the local solvers for DDMs was discussed in~\cite{LNP:2018,LPP:2019}.

Next, we provide the setting of the used parameters.
 We set the parameters for the outer iterations by $\tau=50$ and $\sigma \tau = 1/2$ for the image denoising and inpainting problems, and $\tau = 1, \sigma \tau = 1/2$ for the image segmentation problem.
 For the $O(1/n^2)$ convergent algorithms for the local problems, we set $\gamma = 1/8\tau$, $\tau_0 = 10$, and $\sigma_0 \tau_0 = 1/8$.
 (The same notations for the parameters are used as in~\cite{CP:2011}).
 The full dimension problems ($\mathcal{N} = 1$) are solved by the primal-dual algorithm with the same parameters as the corresponding local problems.
 
 Finally, we note that the local solutions $(u_s^{(n-1)}, \tp_s^{(n-1)})$ from the previous outer iteration were chosen as initial guesses for the local problems to reduce the number of inner iterations.

\subsection{Image denoising}
\begin{figure}[]
\centering
\begin{subfigure}[Noisy ``Cameraman $2048 \times 2048$" (PSNR: 12.07)]
{ \includegraphics[width=3.7cm]{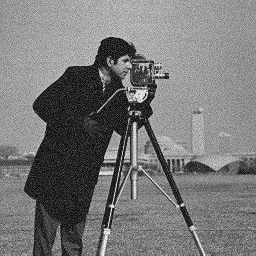} }
\end{subfigure}
\begin{subfigure}[$\mathcal{N} = 1$ (PSNR: 47.68)]
{ \includegraphics[width=3.7cm]{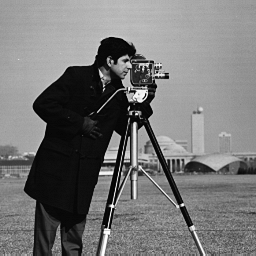} }
\end{subfigure}
\begin{subfigure}[$\mathcal{N} = 16 \times 16$ (PSNR: 48.24)]
{ \includegraphics[width=3.7cm]{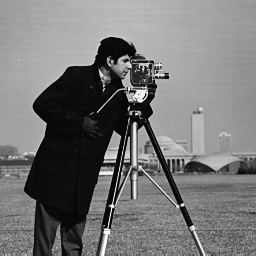} } 
\end{subfigure}

\begin{subfigure}[Noisy ``Boat $2048\times 3072$'' (PSNR: 12.34)]
{ \includegraphics[width=3.7cm]{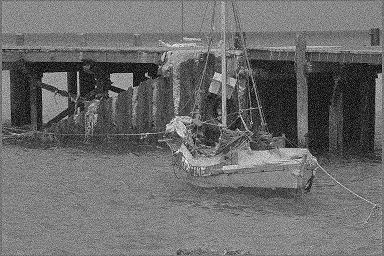} }
\end{subfigure}
\begin{subfigure}[$\mathcal{N} = 1$ (PSNR: 33.82)]
{ \includegraphics[width=3.7cm]{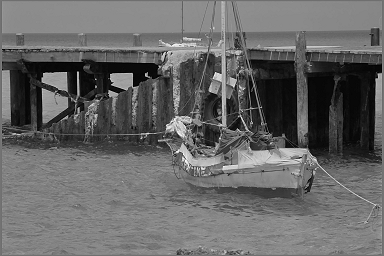} }
\end{subfigure}
\begin{subfigure}[$\mathcal{N} = 16 \times 16$ (PSNR: 33.92)]
{ \includegraphics[width=3.7cm]{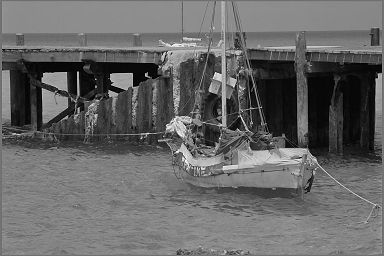} } 
\end{subfigure}

\caption{Results of the proposed method for the image denoising problem with the $L^1$ fidelity term~\eqref{denoising_L1}}
\label{Fig:denoising_L1}
\end{figure}

\begin{figure}[]
\centering
\begin{subfigure}[Cameraman $2048 \times 2048$]
{ \includegraphics[width=5cm]{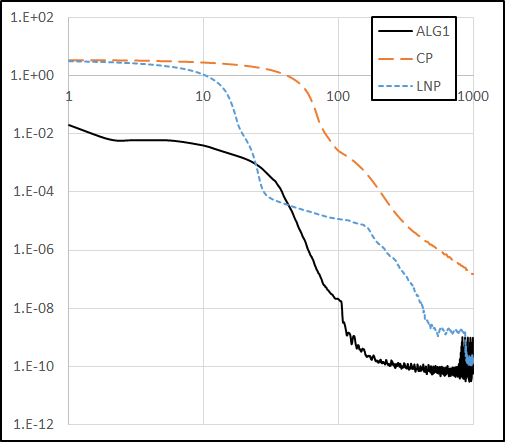} }
\end{subfigure}
\quad\quad
\begin{subfigure}[Boat $2048 \times 3072$]
{ \includegraphics[width=5cm]{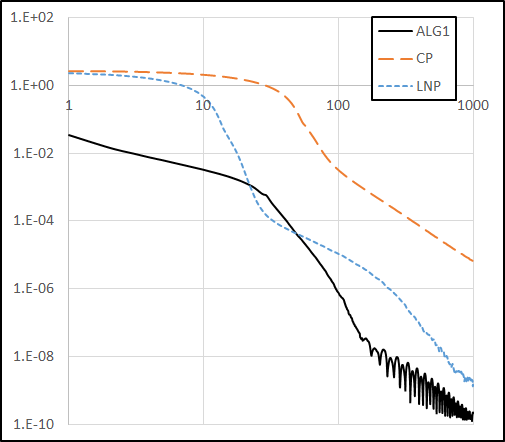} }
\end{subfigure}
\caption{Decay of $\frac{\J (u^{(n)}) - \J (u^*)}{\J (u^*)}$ with respect to the number of iterations~$n$ for various algorithms, applied to the image denoising problem with the $L^1$ fidelity term~\eqref{denoising_L1}}
\label{Fig:L1den_en}
\end{figure}

\begin{table}
\centering
\resizebox{\textwidth}{!}{
\begin{tabular}{| c | c | c c c c | c c c c |} \hline
\multirow{2}{*}{\shortstack{\begin{phantom}1\end{phantom} \\ \begin{phantom}1\end{phantom} \\ Test image}} & 
\multirow{2}{*}{\shortstack{\begin{phantom}1\end{phantom} \\ \begin{phantom}1\end{phantom} \\ $\N$}} &
\multicolumn{4}{c|}{ALG1} & \multicolumn{4}{c|}{LNP} \\ \cline{3-10}
&& PSNR & iter & \begin{tabular}{c}max\\inner\\ iter\end{tabular} & \begin{tabular}{c}wall-clock\\ time (sec)\end{tabular} 
& PSNR & iter & \begin{tabular}{c}max\\inner\\ iter\end{tabular} & \begin{tabular}{c}wall-clock\\ time (sec)\end{tabular} \\
\hline
\multirow{4}{*}{\shortstack{\begin{phantom}1\end{phantom} \\ Cameraman \\ $2048 \times 2048$}}
& 1 & 47.68 & - & 546 & 86.30 					& & & & \\ \cline{2-10}
& $2 \times 2$ & 48.26 & 39 & 266 & 155.85 		& 46.75 & 33 & 1761 & 587.86 \\
& $4 \times 4$ & 48.26 & 45 & 277 & 50.40 		& 46.75 & 61 & 1679 & 141.31 \\
& $8 \times 8$ & 48.25 & 49 & 274 & 14.79 		& 46.78 & 134 & 1796 & 42.37 \\
& $16 \times 16$ & 48.24 & 54 & 279 & 2.19 		& 46.79 & 175 & 1200 & 7.82 \\
\hline
\multirow{4}{*}{\shortstack{\begin{phantom}1\end{phantom} \\ Boat \\ $2048 \times 3072$}}
& 1 & 33.82 & - & 2093 & 476.02 				& & & & \\ \cline{2-10}
& $2 \times 2$ & 33.92 & 61 & 283 & 531.97 		& 33.60 & 43 & 1815 & 970.12 \\
& $4 \times 4$ & 33.92 & 63 & 276 & 165.45 		& 33.62 & 71 & 1823 & 255.70 \\
& $8 \times 8$ & 33.92 & 66 & 258 & 53.68 		& 33.62 & 114 & 1413 & 87.23 \\
& $16 \times 16$ & 33.92 & 71 & 246 & 8.18 		& 33.63 & 161 & 844 & 14.07 \\
\hline
\end{tabular}
}
\caption{Performance of the proposed method for the image denoising problem with the $L^1$ fidelity term~\eqref{denoising_L1}}
\label{Table:L1den}
\end{table}

We present the results of numerical experiments for the $TV$-$L^1$ model for image denoising:
\begin{equation}
\label{denoising_L1}
\min_{u \in X} \left\{ \alpha \| u - f \|_{1, X} + TV(u) \right\}.
\end{equation}
We note that numerical results of the proposed method for ROF model were given in~\cite{LPP:2019}. 
In~\eqref{denoising_L1}, $\alpha F_s (u_s)$ is given by
\begin{equation*}
\alpha F_s (u_s) =  \alpha \| u_s - f \|_{1, X_s}
\end{equation*}
and its proximity operator can be computed as
\begin{equation*}
\prox_{\sigma_0 \alpha F_s} (\bar{u}_s) = f + \shr_{s, \Omega_s} (\bar{u}_s - f, \sigma_0 \alpha),
\end{equation*}
where the shrinkage operator $\shr_{s, S}$ on $S \subset \Omega_s$ is defined by replacing $\Omega$ and $\mathcal{T}$ by $\Omega_s$ and $\mathcal{T}_s$ in~\eqref{shrinkage}, respectively.

We use two test images ``Cameraman $2048 \times 2048$'' and ``Boat $2048 \times 3072$'' with $20\%$ salt-and-pepper noise~(See Fig.~\ref{Fig:denoising_L1}).
PSNR denotes peak signal-to-noise ratio.
We set~$\alpha = 1$.

Thanks to Proposition~\ref{Prop:equiv}, it is able to compare the proposed method with the existing methods for~\eqref{denoising_L1} based on the finite difference discretization.
The following algorithms are used for our performance evaluation:
\begin{itemize}
\item ALG1: Proposed method described in Algorithm~\ref{Alg:DD}, $\N = 8\times 8$.
\item LNP: DDM proposed by Lee, Nam, and Park~\cite{LNP:2018} with the anisotropic total variation, $\N = 8\times 8$, $\tau = 0.05$, $\sigma \tau = 1/8$.
\item CP: Primal-dual algorithm proposed by Chambolle and Pock~\cite{CP:2011} with the anisotropic total variation, $\tau = 10$, $\sigma\tau = 1/8$.
\end{itemize}

In the numerical experiments, the following stopping criterion is used for outer iterations:
\begin{equation}
\label{stop_outer}
\frac{\J (u^{(n)}) - \J (u^*)}{\J (u^*)} < 10^{-5},
\end{equation}
and the following ones are used for inner iterations of ALG1 and LNP, respectively:
\begin{subequations} \begin{eqnarray}
\label{stop_inner1} \frac{\| \p_s^{(n+1)} - \p_s^{(n)} \|_2}{\| \p_s^{(n+1)} \|_2} < 10^{-6}, \\
\label{stop_inner2} \frac{\| u_s^{(n+1)} - u_s^{(n)} \|_2}{\| u_s^{(n+1)} \|_2} < 10^{-6}.
\end{eqnarray} \end{subequations}
Fig.~\ref{Fig:denoising_L1} shows the denoised images obtained by ALG1 with $\N = 16 \times 16$ and CP~($\N = 1$).
In order to highlight the efficiency of the proposed method as a parallel solver, Table~\ref{Table:L1den} shows the performance of the methods with the varying number of subdomains.
In addition, to compare the convergence rate of the proposed method with existing algorithms, we present Fig.~\ref{Fig:L1den_en} which shows decay of $\frac{\J (u^{(n)}) - \J (u^*)}{\J (u^*)}$ for 1000 iterations of three algorithms ALG1, LNP, and CP, where the minimum primal energy~$\J (u^*)$ is computed approximately by $10^6$ iterations of the primal-dual algorithm.

\subsection{Image inpainting}
\begin{figure}[]
\centering
\begin{subfigure}[Corrupted ``Cameraman $2048 \times 2048$" (PSNR: 15.83)]
{ \includegraphics[width=3.7cm]{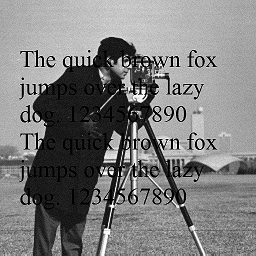} }
\end{subfigure}
\begin{subfigure}[$\mathcal{N} = 1$ (PSNR: 24.96)]
{ \includegraphics[width=3.7cm]{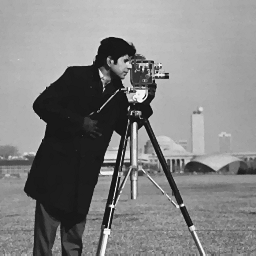} }
\end{subfigure}
\begin{subfigure}[$\mathcal{N} = 16 \times 16$ (PSNR: 25.00)]
{ \includegraphics[width=3.7cm]{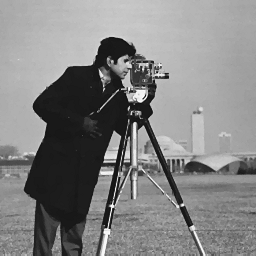} } 
\end{subfigure}

\begin{subfigure}[Corrupted ``Boat $2048\times 3072$'' (PSNR: 16.66)]
{ \includegraphics[width=3.7cm]{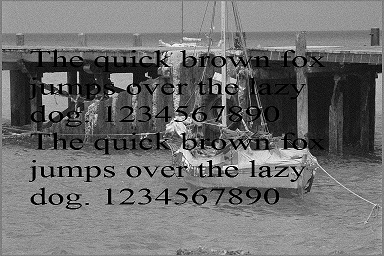} }
\end{subfigure}
\begin{subfigure}[$\mathcal{N} = 1$ (PSNR: 24.32)]
{ \includegraphics[width=3.7cm]{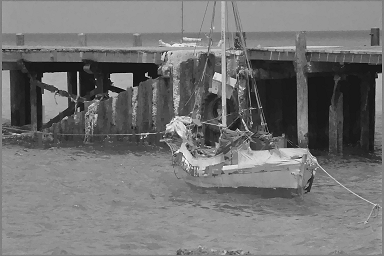} }
\end{subfigure}
\begin{subfigure}[$\mathcal{N} = 16 \times 16$ (PSNR: 24.33)]
{ \includegraphics[width=3.7cm]{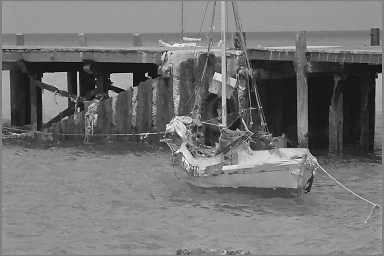} } 
\end{subfigure}

\caption{Results of the proposed method for the image inpainting problem with the $L^2$ fidelity term~\eqref{inpainting_L2}}
\label{Fig:inpainting_L2}
\end{figure}

\begin{figure}[]
\centering
\begin{subfigure}[Cameraman $2048 \times 2048$]
{ \includegraphics[width=5cm]{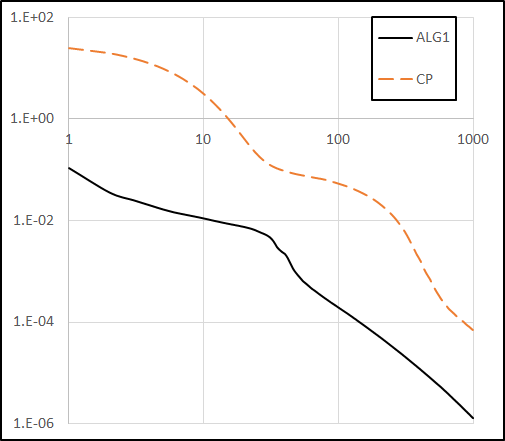} }
\end{subfigure}
\quad\quad
\begin{subfigure}[Boat $2048 \times 3072$]
{ \includegraphics[width=5cm]{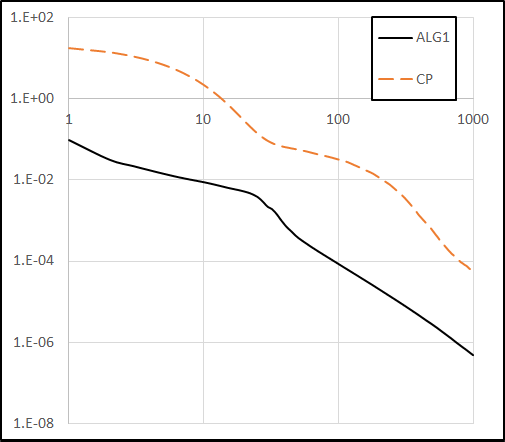} }
\end{subfigure}
\caption{Decay of $\frac{\J (u^{(n)}) - \J (u^*)}{\J (u^*)}$ with respect to the number of iterations~$n$ for various algorithms, applied to the image inpainting problem with the $L^2$ fidelity term~\eqref{inpainting_L2}}
\label{Fig:L2inp_en}
\end{figure}

\begin{table}
\centering
\begin{tabular}{| c | c | c c c c |} \hline
Test image & $\mathcal{N}$ & PSNR & iter & \begin{tabular}{c}max\\inner\\ iter\end{tabular} & \begin{tabular}{c}wall-clock\\ time (sec)\end{tabular} \\
\hline
\multirow{4}{*}{\shortstack{\begin{phantom}1\end{phantom} \\ Cameraman \\ $2048 \times 2048$}}
& 1 & 24.96 & - & 2324 & 389.88 \\ \cline{2-6}
& $2 \times 2$ & 25.00 & 438 & 613 & 805.88 \\
& $4 \times 4$ & 25.00 & 440 & 713 & 260.06 \\
& $8 \times 8$ & 25.00 & 445 & 99 & 88.96 \\
& $16 \times 16$ & 25.00 & 460 & 99 & 13.40 \\
\hline
\multirow{4}{*}{\shortstack{\begin{phantom}1\end{phantom} \\ Boat \\ $2048 \times 3072$}}
& 1 & 24.32 & - & 1861 & 436.32 \\ \cline{2-6}
& $2 \times 2$ & 24.33 & 285 & 817 & 1085.96 \\
& $4 \times 4$ & 24.33 & 286 & 588 & 361.04 \\
& $8 \times 8$ & 24.33 & 289 & 111 & 132.14 \\
& $16 \times 16$ & 24.33 & 298 & 117 & 25.63 \\
\hline
\end{tabular}
\caption{Performance of the proposed method for the image inpainting problem with the $L^2$ fidelity term~\eqref{inpainting_L2}}
\label{Table:L2inp}
\end{table}

We first consider the inpainting model with the $L^2$ fidelity term:
\begin{equation}
\label{inpainting_L2}
\min_{u \in X} \left\{ \frac{\alpha}{2} \| Au - f \|_{2, X}^2 + TV(u) \right\} .
\end{equation}
Here, $A$:~$X \rightarrow X$ is the restriction operator onto $\Omega \setminus D$, so that
its matrix representation is a diagonal matrix whose diagonal entries are either $0$ or $1$.
Thus, computation of the proximity operator of $\alpha F_s (u_s) = \frac{\alpha}{2} \| A_s u_s - f \|_{2, X_s}^2$ is as easy as the case of the denoising problem.
Indeed, we have
\begin{equation*}
\left( \prox_{\sigma_0 \alpha F_s} (\bar{u}_s) \right)_T =
\begin{cases}
\frac{ ( \bar{u}_s )_T + \sigma_0 \alpha (f)_T}{1 + \sigma_0 \alpha} & \textrm{ if } T \subset \Omega_s \setminus D , \\
( \bar{u}_s )_T & \textrm{ if } T \subset D,
\end{cases}
\hspace{0.2cm} T \in \mathcal{T}_s.
\end{equation*}

Two test images ``Cameraman $2048 \times 2048$'' and ``Boat $2048 \times 3072$'' corrupted by additive Gaussian noise with mean $0$ and variance $0.05$ and a text mask are used for numerical experiments~(See Fig.~\ref{Fig:inpainting_L2}).
The model parameter $\alpha$ is set as $\alpha = 10$.
Table~\ref{Table:L2inp} shows the computation results of the proposed method for the varying number of subdomains.
Stop conditions~\eqref{stop_outer} and~\eqref{stop_inner1} are used for outer iterations and inner iterations of ALG1, respectively.
Fig.~\ref{Fig:inpainting_L2} shows the resulting images of ALG1 with $\N = 16 \times 16$ and CP.
Fig.~\ref{Fig:L2inp_en} shows decay of $\frac{\J (u^{(n)}) - \J (u^*)}{\J (u^*)}$ for two algorithms ALG1 and CP.
Here, $\J (u^*)$ is computed by $10^6$ iterations of the primal-dual algorithm.

\begin{figure}[]
\centering
\begin{subfigure}[Corrupted ``Cameraman $2048 \times 2048$" (PSNR: 11.41)]
{ \includegraphics[width=3.7cm]{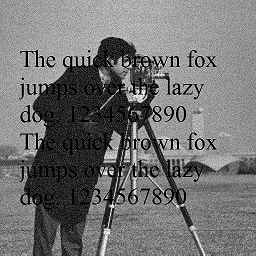} }
\end{subfigure}
\begin{subfigure}[$\mathcal{N} = 1$ (PSNR: 34.70)]
{ \includegraphics[width=3.7cm]{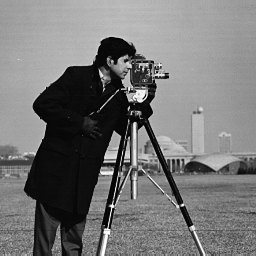} }
\end{subfigure}
\begin{subfigure}[$\mathcal{N} = 16 \times 16$ (PSNR: 35.53)]
{ \includegraphics[width=3.7cm]{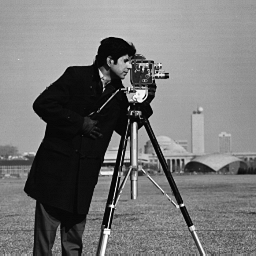} } 
\end{subfigure}

\begin{subfigure}[Corrupted ``Boat $2048\times 3072$'' (PSNR: 11.92)]
{ \includegraphics[width=3.7cm]{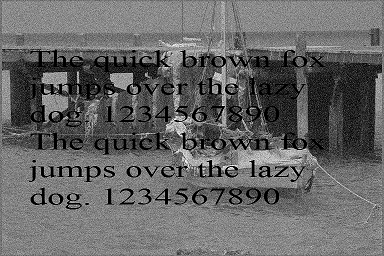} }
\end{subfigure}
\begin{subfigure}[$\mathcal{N} = 1$ (PSNR: 30.83)]
{ \includegraphics[width=3.7cm]{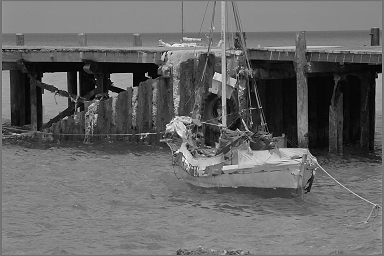} }
\end{subfigure}
\begin{subfigure}[$\mathcal{N} = 16 \times 16$ (PSNR: 31.07)]
{ \includegraphics[width=3.7cm]{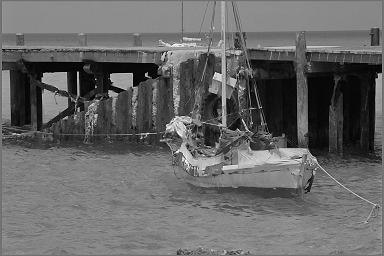} } 
\end{subfigure}

\caption{Results of the proposed method for the image inpainting problem with the $L^1$ fidelity term~\eqref{inpainting_L1}}
\label{Fig:inpainting_L1}
\end{figure}

\begin{figure}[]
\centering
\begin{subfigure}[Cameraman $2048 \times 2048$]
{ \includegraphics[width=5cm]{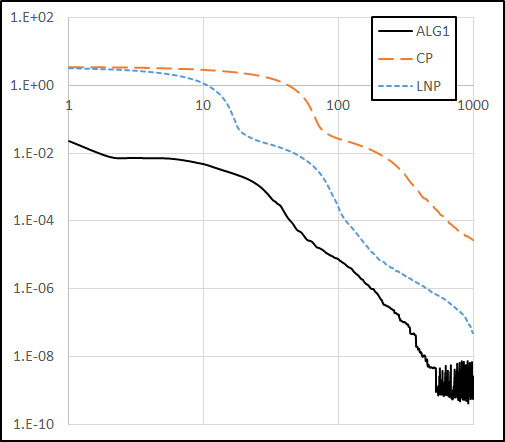} }
\end{subfigure}
\quad\quad
\begin{subfigure}[Boat $2048 \times 3072$]
{ \includegraphics[width=5cm]{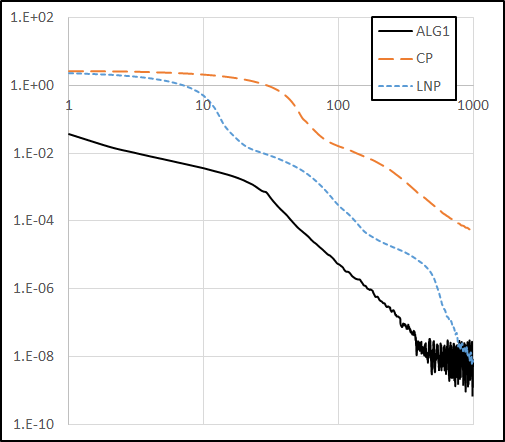} }
\end{subfigure}
\caption{Decay of $\frac{\J (u^{(n)}) - \J (u^*)}{\J (u^*)}$ with respect to the number of iterations~$n$ for various algorithms, applied to the image inpainting problem with the $L^1$ fidelity term~\eqref{inpainting_L1}}
\label{Fig:L1inp_en}
\end{figure}

\begin{table}
\centering
\resizebox{\textwidth}{!}{
\begin{tabular}{| c | c | c c c c | c c c c |} \hline
\multirow{2}{*}{\shortstack{\begin{phantom}1\end{phantom} \\ \begin{phantom}1\end{phantom} \\ Test image}} & 
\multirow{2}{*}{\shortstack{\begin{phantom}1\end{phantom} \\ \begin{phantom}1\end{phantom} \\ $\N$}} &
\multicolumn{4}{c|}{ALG1} & \multicolumn{4}{c|}{LNP} \\ \cline{3-10}
&& PSNR & iter & \begin{tabular}{c}max\\inner\\ iter\end{tabular} & \begin{tabular}{c}wall-clock\\ time (sec)\end{tabular} 
& PSNR & iter & \begin{tabular}{c}max\\inner\\ iter\end{tabular} & \begin{tabular}{c}wall-clock\\ time (sec)\end{tabular} \\
\hline
\multirow{4}{*}{\shortstack{\begin{phantom}1\end{phantom} \\ Cameraman \\ $2048 \times 2048$}}
& 1 & 34.70 & - & 1693 & 304.05 					& & & & \\ \cline{2-10}
& $2 \times 2$ & 35.56 & 78 & 921 & 615.5 			& 34.18 & 173 & 1530 & 1132.71 \\
& $4 \times 4$ & 35.53 & 86 & 958 & 244.21 			& 34.24 & 190 & 1672 & 323.98 \\
& $8 \times 8$ & 35.53 & 90 & 274 & 68.44			& 34.26 & 205 & 1796 & 97.65 \\
& $16 \times 16$ & 35.53 & 101 & 279 & 16.10 		& 34.37 & 239 & 1200 & 24.67 \\
\hline
\multirow{4}{*}{\shortstack{\begin{phantom}1\end{phantom} \\ Boat \\ $2048 \times 3072$}}
& 1 & 30.83 & - & 2262 & 604.34 					& & & & \\ \cline{2-10}
& $2 \times 2$ & 31.08 & 80 & 860 & 1283.93 		& 30.53 & 368 & 1703 & 2134.45 \\
& $4 \times 4$ & 31.08 & 84 & 765 & 432.68 			& 30.55 & 374 & 1888 & 635.88 \\
& $8 \times 8$ & 31.08 & 88 & 258 & 128.09 			& 30.56 & 384 & 1413 & 217.33 \\
& $16 \times 16$ & 31.07 & 97 & 246 & 31.95 		& 30.58 & 402 & 844 & 60.68 \\
\hline
\end{tabular}
}
\caption{Performance of the proposed method for the image inpainting problem with the $L^1$ fidelity term~\eqref{inpainting_L1}}
\label{Table:L1inp}
\end{table}

Now, we consider the following $L^1$ inpainting model:
\begin{equation}
\label{inpainting_L1}
\min_{u \in X} \left\{ \alpha \| Au - f \|_{1, X} + TV(u) \right\}.
\end{equation}
Similarly to the $L^2$ inpainting problem, the proximity operator of $\alpha F_s (u_s) = \alpha \| A_s u_s -f \|_{1, X}$ is given by
\begin{equation*}
\prox_{\sigma_0 \alpha F_s} (\bar{u}_s) = f + \shr_{s, \Omega_s \setminus D} (\bar{u}_s - f, \sigma_0 \alpha ).
\end{equation*}
Test images are corrupted by $20\%$ salt-and-pepper noise and the same text mask as the $L^2$ inpainting problem~~(See Fig.~\ref{Fig:inpainting_L1}).
We use $\alpha = 1$ as the model parameter.
Numerical results of ALG1, LNP, and CP for~\eqref{inpainting_L1} are given in Table~\ref{Table:L1inp}.
Both ALG1 and LNP uses~\eqref{stop_outer} as the stopping criterion for outer iterations.
Stopping criteria for local problems are given in~\eqref{stop_inner1} and~\eqref{stop_inner2} for ALG1 and LNP, respectively.
The recovered images obtained by ALG1 and CP are given in Fig.~\ref{Fig:inpainting_L1}.
Fig.~\ref{Fig:L1inp_en} shows decay of the value of~$\frac{\J (u^{(n)}) - \J (u^*)}{\J (u^*)}$ for three algorithms ALG1, LNP, and CP.

\subsection{Image segmentation}
\begin{figure}[]
\centering
\begin{subfigure}[Cameraman $2048 \times 2048$]
{ \includegraphics[width=3.7cm]{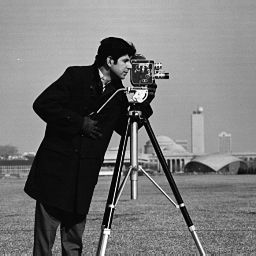} }
\end{subfigure}
\begin{subfigure}[$\mathcal{N} = 1$]
{ \includegraphics[width=3.7cm]{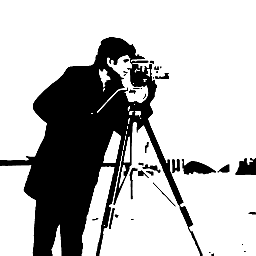} }
\end{subfigure}
\begin{subfigure}[$\mathcal{N} = 16 \times 16$]
{ \includegraphics[width=3.7cm]{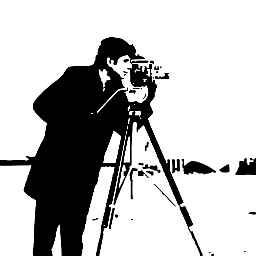} } 
\end{subfigure}

\caption{Results of the proposed method for the image segmentation problem~\eqref{segmentation}}
\label{Fig:segmentation}
\end{figure}

\begin{figure}[]
\centering
\includegraphics[width=5cm]{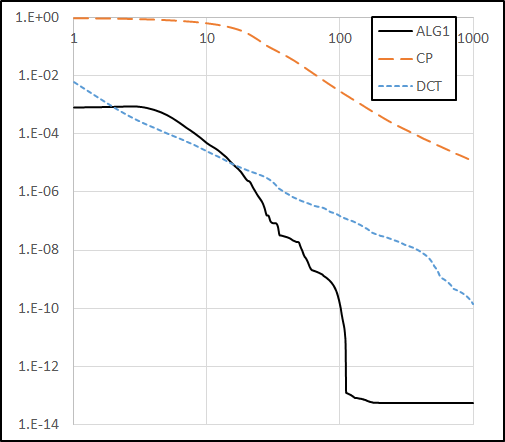}
\caption{Decay of $\frac{\J (u^{(n)}) - \J (u^*)}{| \J (u^*) |}$ with respect to the number of iterations~$n$ for various algorithms, applied to the image segmentation problem~\eqref{segmentation}}
\label{Fig:seg_en}
\end{figure}

\begin{table}
\centering
\resizebox{\textwidth}{!}{
\begin{tabular}{| c | c | c c c | c c c |} \hline
\multirow{2}{*}{\shortstack{\begin{phantom}1\end{phantom} \\ \begin{phantom}1\end{phantom} \\ Test image}} & 
\multirow{2}{*}{\shortstack{\begin{phantom}1\end{phantom} \\ \begin{phantom}1\end{phantom} \\ $\N$}} &
\multicolumn{3}{c|}{ALG1} & \multicolumn{3}{c|}{DCT} \\ \cline{3-8}
&& iter & \begin{tabular}{c}max\\inner\\ iter\end{tabular} & \begin{tabular}{c}wall-clock\\ time (sec)\end{tabular} 
& iter & \begin{tabular}{c}max\\inner\\ iter\end{tabular} & \begin{tabular}{c}wall-clock\\ time (sec)\end{tabular} \\
\hline
\multirow{4}{*}{\shortstack{\begin{phantom}1\end{phantom} \\ Cameraman \\ $2048 \times 2048$}}
& 1 & - & 1070 & 172.47 					& & &  \\ \cline{2-8}
& $2 \times 2$ & 16 & 258 & 153.72 		& 14 & 290 & 85.67 \\
& $4 \times 4$ & 16 & 63 & 37.58 			& 15 & 52 & 18.21 \\
& $8 \times 8$ & 16 & 55 & 6.65 			& 15 & 52 & 5.31 \\
& $16 \times 16$ & 16 & 56 & 1.10 			& 16 & 55 & 1.43 \\
\hline
\end{tabular}
}
\caption{Performance of the proposed method for the image segmentation problem~\eqref{segmentation}}
\label{Table:seg}
\end{table}

As we mentioned in Sect.~\ref{Sec:Model}, the convex Chan--Vese model for the image segmentation model is represented as
\begin{equation}
\label{segmentation}
\min_{u \in X} \left\{ \alpha \left< u , g \right>_X + \chi_{\left\{ 0 \leq u \leq 1\right\}} (u) + TV(u) \right\},
\end{equation}
where $g = (f-c_1)^2 - (f-c_2)^2$.
The proximity operator of $\alpha F_s (u_s) = \left< u , g \right>_X + \chi_{\left\{ 0 \leq u \leq 1\right\}} (u)$ is computed as
\begin{equation*}
\prox_{\sigma_0 \alpha F_s} (\bar{u}_s) = \proj_{\left\{ 0 \leq \cdot \leq 1\right\}} (\bar{u}_s - \sigma_0 \alpha g).
\end{equation*}
As shown in Fig.~\ref{Fig:segmentation}, we use a test image ``Cameraman $2048 \times 2048$.''
We set the model parameters $\alpha = 10$, $c_1 = 0.6$, and $c_2 = 0.1$ heuristically.
Also, to convert the results to binary functions, we use a threshold parameter $1/2$.
Numerical results of ALG1, DCT, and CP for~\eqref{segmentation} are presented in Table~\ref{Table:seg}, where DCT denotes the following algorithm:
\begin{itemize}
\item DCT: DDM proposed by Duan, Chang, and Tai~\cite{DCT:2016} with the anisotropic total variation, $\N = 8 \times 8$, $\tau = 1$, $\sigma \tau = 1/8$.
\end{itemize}
Both ALG1 and DCT uses~\eqref{stop_outer} as a stop condition for outer iterations.
For stop conditions for local problems, ALG1 uses~\eqref{stop_inner1} while DCT uses~\eqref{stop_inner2}.
The segmentation results obtained by ALG1 and CP are provided in Fig.~\ref{Fig:segmentation}.
Fig.~\ref{Fig:seg_en} presents the energy decay of several algorithms containing the proposed method for~\eqref{segmentation}.
Since the minimum primal energy~$\J (u^*)$ is negative in this case, we plot the value of~$\frac{\J (u^{(n)}) - \J (u^*)}{|\J (u^*)|}$, where $\J (u^*)$ is computed by $10^6$ iterations of the primal-dual algorithm.

\subsection{Discussion}
As shown in Figs.~\ref{Fig:denoising_L1}, \ref{Fig:inpainting_L2}, \ref{Fig:inpainting_L1}, and \ref{Fig:segmentation}, the results of the full dimension problem and the proposed method are not visually distinguishable.
Note that the resulting images of the proposed method show no trace of the subdomain interfaces.
In Tables~\ref{Table:L1den}--\ref{Table:seg}, one can observe that solutions for different values of $\N$ have different PSNRs.
It means that the algorithms converge to different solutions.
This is because~\eqref{model_old} admits nonunique solutions in general.
However, as the stopping criterion~\eqref{stop_outer} indicates, their primal energies tend to the minimum.

In Tables~\ref{Table:L1den}--\ref{Table:seg}, one can see that the number of outer iterations monotonically increases as the number of subdomains~$\N$ increases.
We also observe that the numbers of maximum inner iterations of the proposed method are much smaller than the numbers of iterations of the full dimension problem.
The reason is that we utilize more accelerated solvers for the local problems than standard ones.
Reduction of the numbers of inner iterations makes the proposed method faster.

For every problem, we see that the wall-clock time is decreasing as the number of subdomains $\mathcal{N}$ grows.
In particular, large scale images such as ``Cameraman $2048 \times 2048$'' and ``Boat $2048 \times 3072$'' can be processed in a minute with sufficiently many subdomains, while it takes quite a long time with a single domain.
This shows efficiency of the proposed method as a parallel solver for image processing.

Finally, as shown in Figs.~\ref{Fig:L1den_en}, \ref{Fig:L2inp_en}, \ref{Fig:L1inp_en}, and \ref{Fig:seg_en}, the proposed method  shows good performance in terms of the decay rate of the primal energy~$\J (u^{(n)})$.
To be more precise, the primal energy of the proposed method decreases much faster than the one of CP for the full-dimension problem.
The proposed method also outperforms LNP which is a recently developed DDM for $TV$-$L^1$ problems.
Even more, for the segmentation problem, the convergence rate of the proposed method seems to be faster than $O(1/n)$.
However, a theoretical evidence for such fast convergence is still missing.
It is observed in Figs.~\ref{Fig:L1den_en}, \ref{Fig:L2inp_en}, \ref{Fig:L1inp_en}, and \ref{Fig:seg_en} that the primal energy of the proposed method becomes stagnant when the relative error is sufficiently small.
This is due to that local problems are solved inexactly by iterative methods in each iteration.
Such stagnation of the primal energy is not problematic in practice because the quality of the recovered image becomes acceptable enough before the stagnation starts.
Meanwhile, in terms of wall-clock time, the proposed method outperforms LNP for both denoising and inpainting problems while it shows the similar performance to DCT for the segmentation problem.

\section{Conclusion}
\label{Sec:Conclusion}
In this paper, we generalized the primal-dual DDM for the ROF model proposed in~\cite{LPP:2019} to more general total variation minimization problem.
The Fenchel--Rockafellar dual of the model problem was considered.
We constructed the constrained minimization problem which has domain decomposition structure and is equivalent to the dual model problem.
The constrained minimization problem was converted to the equivalent saddle point problem by the method of Lagrange multipliers.
The resulting saddle point problem was solved by the first order primal-dual algorithm and convergence of the dual solution was guaranteed.
We also proved convergence of the primal solution.
Numerical results showed that the proposed method is superior to the existing methods in the sense of convergence rate, and is much faster with sufficiently many subdomains than the full dimension problem.

Even though the proposed DDM is applicable for various total variation regularized problems, we point out that the proposed method is not appropriate for the image deconvolution problem.
We will develop a DDM with similar strategy which is applicable for the image deconvolution problem later.


\begin{thebibliography}{}
\bibitem{Bartels:2012}
Bartels, S.: Total variation minimization with finite elements: convergence and iterative solution. SIAM J. Numer. Anal. \textbf{50}(3), 1162--1180 (2012)
\bibitem{CLL:2011}
Chambolle, A., Levine, S.E., Lucier, B.J.: An upwind finite-difference method for total variation-based image smoothing. SIAM J. Imaging Sci. \textbf{4}(1), 277--299 (2011)
\bibitem{CP:2011}
Chambolle, A., Pock, T.: A first-order primal-dual algorithm for convex problems with applications to imaging. J. Math. Imaging Vis. \textbf{40}(1), 120--145 (2011)
\bibitem{CP:2016}
Chambolle, A., Pock, T.: An introduction to continuous optimization for imaging. Acta Numer. \textbf{25}, 161--319 (2016)
\bibitem{CE:2005}
Chan, T.F., Esedoglu, S.: Aspects of total variation regularized ${L}^1$ function approximation. SIAM J. Appl. Math. \textbf{65}(5), 1817--1837 (2005)

\bibitem{CEN:2006}
Chan, T.F., Esedoglu, S., Nikolova, M.:Algorithms for finding global minimizers of image segmentation and denoising models. SIAM J. Appl. Math. \textbf{66}(5), 1632--1648 (2006)
\bibitem{CW:1998}
Chan, T.F., Wong, C.-K.: Total variation blind deconvolution. IEEE Trans. Image Process. \textbf{7}(3), 370--375 (1998)
\bibitem{CTWY:2015}
Chang, H., Tai, X.-C., Wang, L.-L., Yang, D.: Convergence rate of overlapping domain decomposition methods for the {R}udin-{O}sher-{F}atemi model based on a dual formulation. SIAM J. Imaging Sci. \textbf{8}(1), 564--591 (2015)
\bibitem{DHN:2009}
Dong, Y., Hinterm{\"u}ller, M., Neri, M.: An efficient primal-dual method for ${L}^1 {T}{V}$ image restoration. SIAM J. Imaging Sci. \textbf{2}(4), 1168--1189 (2009)
\bibitem{DCT:2016}
Duan, Y., Chang, H., Tai, X.-C.: Convergent non-overlapping domain decomposition methods for variational image segmentation. J. Sci. Comput. \textbf{69}(2), 532--555 (2016)

\bibitem{FR:1991}
Farhat, C., Roux, F.-X.: A method of finite element tearing and interconnecting and its parallel solution algorithm. Int. J. Numer. Methods Eng. \textbf{32}(6), 1205--1227 (1991)
\bibitem{Fornasier:2007}
Fornasier, M.: Domain decomposition methods for linear inverse problems with sparsity constraints. Inverse Probl. \textbf{23}(6), 2505--2526 (2007)
\bibitem{FLS:2010}
Fornasier, M., Langer, A., Sch{\"o}nlieb, C.-B.: A convergent overlapping domain decomposition method for total variation minimization. Numer. Math. \textbf{116}(4), 645--685 (2010)
\bibitem{FS:2009}
Fornasier, M., Sch{\"o}nlieb, C.-B.: Subspace correction methods for total variation and $l1$-minimization. SIAM J. Numer. Anal. \textbf{47}(5), 3397--3428 (2009)
\bibitem{HHSNW:2018}
Hermann, M., Herzog, R., Schmidt, S., Vidal-N\'{u}\~{n}ez, J., Wachsmuth, G.: Discrete total variation with finite elements and applications to imaging. J. Math. Imaging Vis. \textbf{61}(4), 411--431 (2019)

\bibitem{HL:2013}
Hinterm{\"u}ller, M., Langer, A.: Subspace correction methods for a class of nonsmooth and nonadditive convex variational problems with mixed ${L}^1$/${L}^2$ data-fidelity in image processing. SIAM J. Imaging Sci. \textbf{6}(4), 2134--2173 (2013)
\bibitem{HL:2015}
Hinterm{\"u}ller, M., Langer, A.: Non-overlapping domain decomposition methods for dual total variation based image denoising. J. Sci. Comput. \textbf{62}(2), 456--481 (2015)
\bibitem{LLWY:2016}
Lee, C.-O., Lee, J.H., Woo, H., Yun, S.: Block decomposition methods for total variation by primal--dual stitching. J. Sci. Comput. \textbf{68}(1), 273--302 (2016)
\bibitem{LN:2017}
Lee, C.-O., Nam, C.: Primal domain decomposition methods for the total variation minimization, based on dual decomposition. SIAM J. Sci. Comput. \textbf{39}(2), B403--B423 (2017)
\bibitem{LNP:2018}
Lee, C.-O., Nam, C., Park, J.: Domain decomposition methods using dual conversion for the total variation minimization with ${L}^1$ fidelity term. J. Sci. Comput. \textbf{78}(2), 951--970 (2019)

\bibitem{LPP:2019}
Lee, C.-O., Park, E.-H., Park, J.: A finite element approach for the dual Rudin--Osher--Fatemi model and its nonoverlapping domain decomposition methods. SIAM J. Sci. Comput. \textbf{41}(2), B205--B228 (2019)
\bibitem{Nikolova:2004}
Nikolova, M.: A variational approach to remove outliers and impulse noise. J. Math. Imaging Vis. \textbf{20}(1), 99--120 (2004)
\bibitem{Rockafellar:2015}
Rockafellar, R.T.: Convex Analysis. Princeton University Press, New Jersey (2015)
\bibitem{ROF:1992}
Rudin, L.I., Osher, S., Fatemi, E.: Nonlinear total variation based noise removal algorithms. Physica D. \textbf{60}(1-4), 259--268 (1992)
\bibitem{CS:2002}
Shen, J., Chan, T.F.: Mathematical models for local nontexture inpaintings. SIAM J. Appl. Math. \textbf{62}(3), 1019--1043 (2002)

\bibitem{SC:2003}
Strong, D., Chan, T.F.: Edge-preserving and scale-dependent properties of total variation regularization. Inverse Probl. \textbf{19}(6), S165--S187 (2003)
\bibitem{WL:2011}
Wang, J., Lucier, B.J.: Error bounds for finite-difference methods for Rudin--Osher--Fatemi image smoothing. SIAM J. Numer. Anal. \textbf{49}(2), 845--868 (2011)
\bibitem{WYYZ:2008}
Wang, Y., Yang, J., Yin, W., Zhang, Y.: A new alternating minimization algorithm for total variation image reconstruction. SIAM J. Imaging Sci. \textbf{1}(3), 248--272 (2008)
\bibitem{YK:1996}
You, Y.-L., Kaveh, M.: A regularization approach to joint blur identification and image restoration. IEEE Trans. Image Process. \textbf{5}(3), 416--428 (1996)
\end{thebibliography}

\end{document}